\documentclass[a4paper,12pt]{article}
\author{Johan Bredberg}
\title{On large gaps between consecutive zeros, on the critical line, of some Dirichlet $L$-functions}
\usepackage{amsmath, amssymb, amsthm, cite}
\newtheorem{theorem}{Theorem}
\newtheorem{corollary}{Corollary}
\newtheorem{lemma}{Lemma}
\theoremstyle{definition}
\newtheorem{remark}{Remark}
\theoremstyle{definition}
\newtheorem{definition}{Definition}
\begin{document}
\date{ }
\maketitle

\textbf{Abstract.} 
The aim of this text is to show the existence of large (3.54 times the average) gaps between consecutive zeros, on the critical line, of some Dirichlet $L$-functions $L(s,\chi)$, with $\chi$ being an even primitive Dirichlet character. 

\section{Introduction}
\subsection{Background and setting} \label{section-1.1}

Consider a Dirichlet $L$-function $L(s,\chi)$, with $\chi$ being a 
primitive Dirichlet character modulo $q$. 
The function $L(s,\chi)$ will have many similar properties to the Riemann 
zeta-function $\zeta(s)$. This is the case if we for example 
focus on the distribution of its zeros. 
It is relatively simple to show that $L(s,\chi)$ has regularly 
spaced zeros on the non-positive real axis, the exact location 
of these so called trivial zeros depending on the sign of $\chi(-1)$ 
(as indicated above though we will later focus on even primitive Dirichlet characters). 
All other zeros $s = \sigma + it$ (the non-trivial ones)
satisfy $0 < \sigma < 1$. 

The Generalized Riemann Hypothesis is a conjecture saying that 
in fact all the non-trivial zeros of $L(s,\chi)$ in the critical strip must lie 
on the critical line, i.e.\ satisfy $\sigma = 1/2$. 
In \cite{Hil} Hilano extends to Dirichlet $L$-functions the ideas used by 
Levinson \cite{levins}, when the latter showed that at least a third of the 
non-trivial zeros of the Riemann zeta-function lie 
on the critical line. However, to this day we do not know the 
whole truth about the horizontal distribution of the zeros.

Another simple question to ask is how many zeros there are up to a certain height in the critical strip. The zeros of a Dirichlet $L$-function are in general not symmetrically distributed with respect to the real axis. Therefore it is natural to define
\[ N(T,\chi) := \#\Big\{s = \sigma + it : L(s,\chi) = 0, 
0 < \sigma < 1 \textrm{ and } |t| < T\Big\} \]
and study how $N(T,\chi)$ behaves as a function of $T$ and of the modulus $q.$ 
By generalising the proof of the corresponding result for the Riemann zeta-function, one may show (see Chapter $16$ in \cite{dav}) that 
for $T \geqslant 2,$ 
\begin{equation*} 
\frac{1}{2}N(T,\chi) = \frac{T}{2\pi} \log(Tq/(2\pi)) - \frac{T}{2\pi} + O(\log(Tq)).
\end{equation*}
One may analogously directly obtain an expression for the number of zeros of $L(\sigma + it, \chi)$ with $T \leqslant t < 2T$ and unsurprisingly the answer is 
\begin{align*}
&\Big\{\frac{2T}{2\pi} \log(2Tq/(2\pi)) - \frac{2T}{2\pi}\Big\} 
- \Big\{\frac{T}{2\pi}\log(Tq/(2\pi)) - \frac{T}{2\pi}\Big\} + O(\log(Tq)) \\
&= \frac{T}{2\pi} \log(Tq) + O(T) + O(\log q).
\end{align*}
This tells us that if we let $T$ be large but fixed and then study $N(T,\chi)$ as a function of $q,$ the average difference of the ordinates of two consecutive zeros at height $T$ is approximately $2\pi/\log q.$ Assuming the Generalised Riemann Hypothesis, Selberg \cite{sel2} proved a result which holds uniformly for all $T$ and all primes $q.$ Denoting by $N_2(T,\chi)$ the number of zeros of $L(s,\chi)$ with $0 < \sigma < 1$ and $0 \leqslant t \leqslant T,$ possible zeros with $t = 0$ or $t = T$ counting one-half only, his result is
\[ N_2(T,\chi) = \frac{T}{2\pi} \log (Tq / (2 \pi)) - \frac{T}{2\pi} 
+ O\Big(\frac{\log (q(1 + T))}{\log \log (q(3 + T))}\Big). \]

One might wonder in more detail what the vertical distribution of 
the zeros of $L(s,\chi)$ looks like. Of course one may ask the same question for the Riemann zeta-function. Many have worked on problems 
in this area. We will briefly go through the history of such results next.

\subsection{Brief history of previous results}

Denote by $\{\gamma_{n}\}$ the sequence of ordinates of all 
zeros of $\zeta(s)$ in the upper halfplane, ordered in 
non-decreasing order. As the distribution of the $\gamma_{n}$ 
becomes denser as we move up in the critical strip, it makes 
more sense to consider the normalised vertical distance between 
consecutive zeros, i.e.\ to ask what we can say about 
\begin{equation*} 
\mu := \liminf_{n \to \infty} 
\frac{\gamma_{n+1}-\gamma_{n}}{(2\pi/\log \gamma_{n})} 
\quad \textrm{and  }  
\lambda := \limsup_{n \to \infty} 
\frac{\gamma_{n+1}-\gamma_{n}}{(2\pi/\log \gamma_{n})}.
\end{equation*} 

In 1946, Selberg \cite{sel} remarked that 
\begin{equation*} 
\mu < 1 < \lambda.
\end{equation*} 
This is the only unconditional result there is. 
Fujii remarks that a similar result holds for Dirichlet $L$-functions in his article \cite{fuj}.

In 1973, Montgomery \cite{Mont} showed that the Riemann Hypothesis\footnote{If the Riemann Hypothesis is false infinitely many times, then trivially there are infinitely many pairs of non-trivial zeros of $\zeta(s)$ with the same ordinate, and hence we would have $\mu = 0.$} implies that 
$\mu < 0.68.$ In the same article he also made, again on the assumption of the truth of the Riemann Hypothesis, his famous pair correlation conjecture for 
the non-trivial zeros of the Riemann zeta-function. If we denote such a zero by $1/2 + i\gamma,$ then it says that for fixed 
$0 < \alpha < \beta,$ 
\[ \sum_{\substack{\gamma, \gamma' \in [0,T] \\ 
2\pi \alpha / \log T \leqslant \gamma - \gamma' \leqslant 2\pi \beta /\log T}} 1 
\sim \frac{T\log T}{2\pi} \int_{\alpha}^{\beta} 
1 - \Big(\frac{\sin(\pi u)}{\pi u}\Big)^2 \, du \]
as $T \to \infty.$ 
This conjecture implies that $\mu = 0.$ 
It is believed that $\lambda = \infty.$ 

Then in 1980, Mueller \cite{mueller} managed to prove, albeit on the assumption of the truth of the Riemann Hypothesis, the existence of what one may actually begin to properly call large gaps. The underlying idea in her proof is nice and easy to explain, so we give a brief sketch here. Let us assume the Riemann Hypothesis and that 
\[ \limsup_{n \to \infty} 
\frac{\gamma_{n+1}-\gamma_{n}}{(2\pi/\log \gamma_{n})} 
< 1.9. \]
We may therefore pick a (sufficiently) large value of $T$ such that all gaps between $T$ and $2T$ are at most $1.9 \cdot 2\pi/\log T.$ Our goal is to show that the latter leads to a contradiction. 

We get that
\begin{equation} \label{mueller-ekv}
\int_{T}^{2T} |\zeta(1/2 + it)|^2 \; dt 
\leqslant \sum_{T < \gamma \leqslant 2T} 
\int_{\gamma - 0.95 \cdot 2 \pi/\log T}^{\gamma 
+ 0.95 \cdot 2 \pi/\log T} |\zeta(1/2 + it)|^2 \; dt.
\end{equation}
Hardy and Littlewood \cite{HL} had already in 1918 showed that the LHS\footnote{Left Hand Side} of (\ref{mueller-ekv}) is asymptotic to $T\log T,$ as $T \to \infty.$ In order to estimate the RHS\footnote{Right Hand Side} of (\ref{mueller-ekv}), Mueller used a result of Gonek (see \cite{Gonek-ett} or \cite{Gonek-tva}), namely that on assuming the Riemann Hypothesis we have 
\begin{equation} \label{gonek-mueller}
\sum_{0 < \gamma \leqslant T} 
\Big|\zeta\Big(\frac{1}{2} + i(\gamma + 2\pi c/\log T)\Big)\Big|^2 
= \frac{T}{2\pi} \log^2{T} 
\Big\{1-\Big(\frac{\sin(\pi c)}{\pi c}\Big)^2\Big\} 
+ O(T\log T)
\end{equation}
as $T \to \infty,$ uniformly for $|c| \leqslant (4\pi)^{-1}\log(\frac{T}{2\pi}).$ 
Clearly we may apply (\ref{gonek-mueller}) twice and take the difference in order to obtain
\begin{equation} \label{gonek-mueller-tva}
\sum_{T < \gamma \leqslant 2T} 
\Big|\zeta\Big(\frac{1}{2} + i(\gamma + 2\pi c/\log T)\Big)\Big|^2 
= \frac{T}{2\pi} \log^2{T} 
\Big\{1-\Big(\frac{\sin(\pi c)}{\pi c}\Big)^2\Big\} 
+ O(T\log T).
\end{equation}
Now we simply integrate (\ref{gonek-mueller-tva}) with respect to $c$ over the interval $[-0.95,0.95]$ and obtain that the RHS of (\ref{mueller-ekv}) is asymptotic (as $T \to \infty$) to 
\[ A T \log T, \]
where 
\[ A = \int_{-0.95}^{0.95} 
\Big\{1-\Big(\frac{\sin(\pi c)}{\pi c}\Big)^2\Big\} \; dc 
\approx 0.9973. \]
Since $0.9973 < 1$ we have a contradiction to (\ref{mueller-ekv}), as desired.

In 1981, Montgomery and Odlyzko \cite{MontOdl} improved the value of $1.9$ to 
$1.9799$ and also found the value $0.5179$ for the lim inf. Their method turned out to be the other side of Mueller's coin, as is explained by Conrey, Ghosh and Gonek \cite{conrey-comp}. The latter article also improved the numerical value of the constants, to $2.337$ and $0.5172$ respectively. 

Hall (see \cite{hall-ett,hall-tva,hall-tre,hall-fyra}) only deals with large gaps and he also only looks at the zeros of the Riemann zeta-function which lie on the critical line (regardless of whether these are all of the non-trivial zeros or not). He then analogously defines lim sup of the gaps between such zeros (this naturally coincides with the usual definition if we assume the Riemann Hypothesis) and for this he obtained the value $2.26$ in \cite{hall-ett}. This was done by using (for details on the basic idea of Hall's proof, the reader is recommended to read Section \ref{section-simple-hall}) the simplest fourth power version of Wirtinger's inequality, namely that if $y(t)$ is a continuously differentiable function satisfying 
$y(0) = y(\pi) = 0$ then 
\[ \int\limits_{0}^{\pi} y(t)^4 \, dt 
\leqslant (4/3) \int\limits_{0}^{\pi} y'(t)^4 \, dt. \]
Hall later \cite{hall-tva} chose to instead work with the Wirtinger-inequality that says that if $y(t)$ is a continuously differentiable function satisfying $y(0) = y(\pi) = 0$ and $v \geqslant 0$ then 
\[ \int\limits_{0}^{\pi} y'(t)^4 
+ 6vy(t)^2y'(t)^2 \, dt 
\geqslant 3\lambda_{0}(v) \int\limits_{0}^{\pi} y(t)^4 \, dt, \]
where 
\[ \lambda_{0}(v) := \frac{1}{8}\{ 1 + 4v + \sqrt{1 + 8v} \}. \]
By applying this inequality to Hardy's function $Z(t)$ with $v = 22/49$ he found the value 
$\sqrt{11/2} \approx 2.34$ for his lim sup. Finally, in 2005, Hall \cite{hall-fyra} applied the simplest version of the Wirtinger's inequality, which says that if $y(t)$ is continuously differentiable and is zero at $t = 0$ and $t = \pi$ then 
\begin{equation} \label{wirtingers-latt}
\int\limits_{0}^{\pi} y(t)^2 \, dt \leqslant \int\limits_{0}^{\pi} y'(t)^2 
\, dt, 
\end{equation} 
to the function 
\[ Z(t)Z(t + 1.315 \cdot 2\pi/\log T) \] 
and obtained the value $2.63$ for his lim sup.

On the assumption of the Riemann Hypothesis, it was shown in 2010 by Feng and Wu \cite{FW} that $\mu < 0.5154$ and $\lambda > 2.7327$ and in 2011 by Bredberg \cite{bredberg} that $\lambda > 2.766.$ 

On the assumption of the Generalised Riemann Hypothesis, it was shown in 2008 by Ng \cite{Ng} that $\lambda > 3$ and this was improved upon in 2011 to 
$\lambda > 3.033$ by Bui \cite{bui}. 

\subsection{Statement of the main result}

The aim of this text is to prove the following theorem:

\begin{theorem}[Main Theorem] \label{huvudsats}
Let $T > 0$ be fixed. For sufficiently large $Q$ there exist $q \in (\frac{5Q}{4},\frac{7Q}{4}]$ for which there exists an even primitive Dirichlet character $\chi (\textrm{\emph{mod} } q)$ such that the function 
$t \mapsto L(\frac{1}{2} + it,\chi)$ has no zeros in some subinterval of $[T,2T]$ of length at least $3.54 \times \frac{2\pi}{\log Q}.$
\end{theorem}

\begin{remark} \label{average-gaps}
Assuming the Generalised Riemann Hypothesis, the above can be compared with the discussion in Section \ref{section-1.1}, where we noted that if $T$ is fixed and large then the relevant average gap-length is $\frac{2\pi}{\log Q}.$
\end{remark}

\subsection{Simple introduction to Hall's method} \label{section-simple-hall}

As mentioned previously, Hall chose to consider only the zeros lying on the critical line. He then unconditionally found a relatively nice lower bound for the analogously defined lim sup. It is such a path we will follow in this text. For the reader's convenience we will therefore now present a brief sketch of the proof of the simplest case where Hall's idea may be applied.  

Let $T$ be large. Rather than working with the Riemann zeta-function we work with Hardy's function defined by 
(see for example Chapter $6$ in Edwards \cite{edward})
\[Z(t) := \exp(i\theta(t)) \zeta(1/2 + it),\]
with
\[ \theta(t) := \textrm{Im}\Big(\log\Big(\Gamma\Big(\frac{1}{4} + \frac{it}{2}\Big)\Big)\Big) - \frac{t}{2} \log{\pi}. \]
Now $Z(t)$ is an even real function for real $t$ and is clearly zero exactly when $\zeta(1/2 + it) = 0.$ 
Label the ordinates of the zeros of $Z(t)$ lying in the interval $[T,2T]$ as $t_1,...,t_N.$ We may make a linear substitution in the simplest version of Wirtinger's inequality (see (\ref{wirtingers-latt})) so that the role of the points $0$ and $\pi$ are overtaken by two general points, say $a$ and $b.$ Applying the outcoming result to $Z(t)$ between two consecutive zeros $t_i$ and $t_{i+1}$ yields
\begin{equation} \label{hall-inl}
\int\limits_{t_{i}}^{t_{i+1}} Z(t)^2 \, dt 
\leqslant \Big( \frac{t_{i+1} - t_{i}}{\pi} \Big)^2
\int\limits_{t_{i}}^{t_{i+1}} Z'(t)^2 \, dt 
\leqslant \Big( \frac{2\sqrt{3}}{\log T} \Big)^2
\int\limits_{t_{i}}^{t_{i+1}} Z'(t)^2 \, dt,
\end{equation}
where the latter step is true IF we assume that all gaps in the interval $[T,2T]$ are at most $\sqrt{3} \cdot \frac{2\pi}{\log T}.$ However, at this point we remark that it is our goal to show the existence of a gap of length at least $\sqrt{3} \cdot \frac{2\pi}{\log T}$ lying in $[T,2T].$ Since we are thus happy if our above assumption fails, we assume that it holds. Then we proceed by summing up the inequality (\ref{hall-inl}) for all pairs of consecutive zeros of $Z(t)$ in $[T,2T].$ If the zeros near the endpoints were not very close to the endpoints that would a priori give us a large gap, thus we may assume that $t_1$ is approximately $T$ and that $t_N$ is approximately $2T.$ We basically obtain that  
\begin{equation} \label{hall-inl-olikhet}
\int\limits_{T}^{2T} Z(t)^2 \, dt 
\leqslant \Big( \frac{2\sqrt{3}}{\log T} \Big)^2
\int\limits_{T}^{2T} Z'(t)^2 \, dt.
\end{equation}
However, it is easy to show that
\[ \int\limits_{T}^{2T} Z(t)^2 \, dt \sim T\log T \]
and that 
\[ \int\limits_{T}^{2T} Z'(t)^2 \, dt \sim T(\log T)^3/12. \]
Had we above assumed that all the relevant gaps were at most \emph{slightly less} than $\sqrt{3} \cdot \frac{2\pi}{\log T},$ we would now have obtained a contradiction when we put in these two well-known estimates into our inequality (\ref{hall-inl-olikhet}). Thus our goal has been reached.

\subsection{Overview of the remainder of this text}

We will here give a brief overview of the remainder of this text. The hope is to quickly help the reader to understand what the big picture is. It is crucial that the reader has read Section \ref{section-simple-hall} which describes the simplest case in which Hall's idea may be applied. For more details and proofs, the reader is of course welcome to read the full text.

We suppose (see Main Assumption in Section \ref{section-kap-fem}) that for all 
$\frac{5Q}{4} < q \leqslant \frac{7Q}{4}$ and all even primitive Dirichlet characters $\chi \textrm{(mod } q)$ we have that all the gaps lying in the interval $[T,2T]$ between consecutive zeros of the function 
$t \mapsto L(\frac{1}{2} + it,\chi)$ are at most $\frac{3\kappa}{\log Q}$ 
($\kappa$ will in fact be \emph{defined} to be the smallest number satisfying this). 
We will prove Theorem \ref{huvudsats} by showing that it is necessary that $\kappa > 1.18 \times \frac{2\pi}{\log Q}.$ 

Similarly to how we in Section \ref{section-simple-hall} introduced the function 
$Z(t):\mathbb{R} \to \mathbb{R}$ whose zeros coincided with the zeros of 
$t \mapsto \zeta(1/2 + it),$ we introduce in Section \ref{section-kap-fem} the function $W(t,\chi):\mathbb{R} \to \mathbb{R}$ whose zeros coincide with the zeros of $t \mapsto L(1/2 + it, \chi).$ The latter implies that all the gaps lying in the interval $[T,2T]$ between consecutive zeros of the function $W(t,\chi)$ are at most $\frac{3\kappa}{\log Q}.$ A simple proof by contradiction\footnote{One needs to be careful close to the endpoints. For strict correctness, the reader is directed to Section \ref{section-kap-fem}.} shows that basically all the gaps lying in the interval $[T,2T]$ between consecutive zeros of the function defined by
\[ f(t,\chi,\kappa) := W(t - \frac{\kappa}{\log Q}, \chi) W(t,\chi) 
W(t + \frac{\kappa}{\log Q}, \chi) \]
are at most $\frac{\kappa}{\log Q}.$

In analogy with Section \ref{section-simple-hall}, we proceed by applying the simplest version of Wirtinger's inequality to $f(t,\chi,\kappa).$ The result is Corollary \ref{wirtinger-ineq-scaled-sum-sum}, which says that 
\begin{align} \label{cor-4-i-inl}
&\displaystyle\sum_{\frac{5Q}{4} < q \leqslant \frac{7Q}{4}} \; \;
\sideset{}{^\flat}\sum_{\chi \textrm{(mod } q)} \;
\int\limits_{t_{1,\chi}}^{t_{N_{\chi},\chi}} f(t,\chi,\kappa)^2 \, dt \nonumber \\
&\leqslant \Big( \frac{\kappa}{\pi \log Q} \Big)^2 
\displaystyle\sum_{\frac{5Q}{4} < q \leqslant \frac{7Q}{4}} \; \;
\sideset{}{^\flat}\sum_{\chi \textrm{(mod } q)} \;
\int\limits_{t_{1,\chi}}^{t_{N_{\chi},\chi}} f'(t,\chi,\kappa)^2 \, dt. 
\end{align} 
Here $t_{1,\chi}$ (which is approximately $T$) and 
$t_{N_{\chi},\chi}$ (which is approximately $2T$) are roughly speaking the first respectively the last zero of $f(t,\chi,\kappa)$ in the interval $[T,2T]$ and we have summed up the resulting inequality over all even primitive Dirichlet characters modulo $q$ and then have also in the outer sum summed over all 
$\frac{5Q}{4} < q \leqslant \frac{7Q}{4}.$

Basically we estimate both sides\footnote{Actually we only find a lower bound for the LHS and an upper bound for the RHS, however, there is no real loss here.} of (\ref{cor-4-i-inl}). However, as is to be expected keeping in mind that our function $f(t,\chi,\kappa)$ depends on $\kappa$, these results are expressions in terms of $\kappa.$ By substituting in these into (\ref{cor-4-i-inl}) we obtain an inequality in terms of $\kappa.$ In Section \ref{Diric-kapitel-conclusion} we note that this inequality requires that $\kappa > 1.18 \times \frac{2\pi}{\log Q},$ as desired.

Much of this text concerns the treatment of both sides of (\ref{cor-4-i-inl}). Very roughly speaking what we want to do here is to be able to evaluate sixth power moments of Dirichlet $L$-functions. No one has yet been successful in proving the sixth power moment of the Riemann zeta-function or Dirichlet $L$-functions. However, Conrey, Iwaniec and Soundararajan \cite{CIS} have found a nice expression for the sixth power moment of Dirichlet $L$-functions, \emph{but} only when one averages over all even (or odd) primitive Dirichlet characters modulo $q$ over a range of values of $q$ and they also need to average over an interval in $t.$ For the statement of their main result, see Theorem \ref{sats1} in Section~\ref{section-cis-artikel}.

Their main result is the starting block for the evaluations in this text. In Section \ref{section-kap-tva} we discuss how to interpret their result and go through some simple consequences of it. In Section \ref{section-kap-tre} we obtain a simpler version of their theorem 
(see Theorem \ref{sats1.3}). This has two purposes. Firstly, the calculations of the expressions in terms of $\kappa$ described above are long and complicated and it is simpler to explain them if we start from a simpler theorem. Secondly, we also want to evaluate sixth power moments with differentiated terms in the integrand. Here it is important to us that we have first simplified some terms featuring in their main theorem (we make some terms independent of the ``shifts" and can therefore treat these as constants during differen-tiation). In Section \ref{Diric-kapitel-differentiating-in-cis} our simplified theorem is used together with Cauchy's integral trick (that is to say we use the expression for our function together with Cauchy's integral formula for the derivative of an analytic function to deduce an expression for the derivative of our function) in order to obtain a differentiated version of our simplified theorem. In Section \ref{Diric-kapitel-investigation} we go through the details of how our previous results in practice enable us to use Corollary~\ref{wirtinger-ineq-scaled-sum-sum} to obtain an inequality in terms of $\kappa.$ Finally, in Section \ref{Diric-kapitel-calculation} we carefully illustrate the simplest of a few similar calculations, whose answers are given in Section \ref{Diric-kapitel-investigation-kappa-coefficients}. 

\subsection{On the article ``The sixth power moment of Dirichlet $L$-functions"} \label{section-cis-artikel}

In this text we will thus make use of the main theorem in the article ``The sixth power moment of Dirichlet $L$-functions", written by Conrey, Iwaniec and Soundararajan (see Theorem 1 in \cite{CIS}). This theorem will be reproduced as Theorem \ref{sats1} in this section. However, first there is some notation to be introduced.

Let $A$ and $B$ be sets of complex numbers, with cardinality equal to $3.$ The elements of these two sets are the so called ``shifts".
Suppose that the modulus of the real part of $\alpha$ and $\beta$ is $< 1/4$ for $\alpha \in A$ and $\beta \in B.$ 
If we let 
\[ \Lambda(s,\chi) := (q/\pi)^{(s-1/2)/2} \Gamma(s/2) L(s,\chi), \] 
then 
\begin{align*}
\Lambda_{A,B}(\chi) 
&:= \displaystyle\prod_{\alpha \in A} \Lambda(1/2 + \alpha,\chi)
\displaystyle\prod_{\beta \in B} \Lambda(1/2 + \beta,\bar{\chi}) \\
&= \Big( \frac{q}{\pi} \Big)^{\delta_{A,B}} 
G_{A,B}
\mathcal{L}_{A,B}(\chi), \nonumber
\end{align*}
where 
\begin{equation*} 
\mathcal{L}_{A,B}(\chi)
:= \displaystyle\prod_{\alpha \in A} L(1/2 + \alpha,\chi)
\displaystyle\prod_{\beta \in B} L(1/2 + \beta,\bar{\chi}), 
\end{equation*} 
\begin{equation*} 
G_{A,B}
:= \displaystyle\prod_{\alpha \in A} 
\Gamma\Big( \frac{1/2 + \alpha}{2} \Big)
\displaystyle\prod_{\beta \in B} 
\Gamma\Big( \frac{1/2 + \beta}{2} \Big) 
\end{equation*}
and
\begin{equation*} 
\delta_{A,B}
:= \frac{1}{2}
\Big( \displaystyle\sum_{\alpha \in A} \alpha + 
\displaystyle\sum_{\beta \in B} \beta \Big).
\end{equation*}
Further, let 
\begin{equation*} 
\mathcal{Z}(A,B)
:= \displaystyle\prod_{\substack{\alpha \in A \\ \beta \in B}} 
\zeta(1 + \alpha + \beta)
\end{equation*}
and
\begin{equation*} 
\mathcal{A}(A,B)
:= \displaystyle\prod_{p} \mathcal{B}_{p}(A,B)\mathcal{Z}_{p}(A,B)^{-1}, 
\end{equation*}
with 
\begin{equation} \label{z-cis-def}
\mathcal{Z}_{p}(A,B)
:= \displaystyle\prod_{\substack{\alpha \in A \\ \beta \in B}} 
\zeta_{p}(1 + \alpha + \beta),
\end{equation}
where
\begin{equation*} 
\zeta_{p}(x)
:= \Big( 1 - \frac{1}{p^{x}} \Big)^{-1}
\end{equation*}
and 
\begin{equation} \label{B-def}
\mathcal{B}_{p}(A,B)
:= \int\limits_{0}^{1} 
\displaystyle\prod_{\alpha \in A} z_{p,\theta}(1/2 + \alpha)
\displaystyle\prod_{\beta \in B} z_{p,-\theta}(1/2 + \beta)
\, d\theta,
\end{equation}
with $z_{p,\theta}(x) := 1/(1 - e(\theta)/p^{x}),$ where 
$e(\theta) := \exp(2\pi i \theta).$
The conditions on the real parts of the elements of $A$ and $B$ ensure that the Euler product for $\mathcal{A}$ converges 
absolutely\footnote{See the proof of Lemma \ref{hjalp-lemma-tre} and in particular (\ref{hjalp-l-tre-referens}).}. 
Let 
\[ \mathcal{B}_{q} := \prod_{p|q} \mathcal{B}_{p}. \] 
Next let
\begin{equation} \label{q-stor}
\mathcal{Q}_{A,B}(q)
:= \displaystyle\sum_{\substack{S \subseteq A \\ T \subseteq B \\ |S| = |T|}} 
\mathcal{Q}(\bar{S} \cup (-T),\bar{T} \cup (-S);q),
\end{equation}
where $\bar{S}$ denotes the complement of $S$ in $A$ and by the set $-S$ we mean $\{ -s : s \in S \},$ and 
\begin{equation} \label{curly-q-def}
\mathcal{Q}(X,Y;q):= \Big( \frac{q}{\pi} \Big)^{\delta_{X,Y}} 
G_{X,Y} \Big( \frac{\mathcal{AZ}}{\mathcal{B}_q} \Big)(X,Y).
\end{equation}
For future need we notice the obvious fact that
\begin{equation} \label{curly-q-symmetry}
\mathcal{Q}(X,Y;q) = \mathcal{Q}(Y,X;q).
\end{equation}
For future convenience we recall three more definitions from \cite{CIS}. Let\footnote{We have here chosen to keep the same notation as in \cite{CIS}, although the usage of $a_{3}(\mathcal{L})$ is confusing in that there is no dependence on ``$\mathcal{L}$".} 
\begin{equation} \label{a3-def}
a_{3}
:= \displaystyle\prod_{p} \Big\{ \Big( 1 - \frac{1}{p} \Big)^4 
\Big( 1 + \frac{4}{p} + \frac{1}{p^2} \Big) \Big\}, 
\end{equation}
\begin{equation*} 
a_{3}(\mathcal{L})
:= \displaystyle\prod_{p} \Big\{ \Big( 1 - \frac{1}{p} \Big)^5 
\Big( 1 + \frac{5}{p} - \frac{5}{p^2} + \frac{14}{p^3} 
- \frac{15}{p^4} + \frac{5}{p^5} + \frac{4}{p^6} 
- \frac{4}{p^7} + \frac{1}{p^8} \Big) \Big\}
\end{equation*}
and
\begin{equation*}
A_{t} := \{ \alpha + t : \alpha \in A \}. 
\end{equation*}

They prove the following:

\begin{theorem}[Main theorem in \cite{CIS}] \label{sats1}
Suppose that $|A| = |B| = 3$ and that $\alpha, \beta \ll 1/\log Q,$ for $\alpha \in A$, $\beta \in B.$ Suppose that $\Psi$ is smooth on $\mathbb{R}$ and compactly supported in $[1,2]$ and $\Phi(t)$ is an entire function of t which decays rapidly as $t \to \infty$ in any fixed horizontal strip. Then
\begin{align} \label{sats1-ekv}
&\sum_{q \geqslant 1} \Psi\Big(\frac{q}{Q}\Big) \int\limits_{-\infty}^{\infty} \Phi(t) \sideset{}{^\flat}\sum_{\chi \textrm{\emph{(mod} } q)} \Lambda_{A_{it},B_{-it}}(\chi) \, dt \nonumber \\
&= \sum_{q \geqslant 1} \Psi\Big(\frac{q}{Q}\Big) \int\limits_{-\infty}^{\infty} \Phi(t) \phi^\flat(q) \mathcal{Q}_{A_{it},B_{-it}}(q) \, dt + O(Q^{7/4+\epsilon}),
\end{align}
uniformly in $\alpha$ and $\beta,$ where the $\sideset{}{^\flat}\sum$ indicates that the sum is restricted to even primitive characters and $\phi^\flat(q)$ denotes the number of primitive even Dirichlet characters modulo $q.$
\end{theorem}

\begin{remark}
Note that here the ``big oh"-constant may depend on $\Psi$ and $\Phi.$
\end{remark}

\textbf{Brief sketch of proof:}
For a complete proof of Theorem \ref{sats1}, the reader is referred to \cite{CIS}. Below is a brief sketch of the proof. We will here assume that 
$q > 1,$ which can basically be justified by recalling that $\Psi$ is compactly supported in $[1,2].$

Consider 
\begin{equation} \label{cis-lambda-icke-noll}
\frac{1}{2\pi i} \int_{(1)} 
\Lambda_{A_{s},B_{s}} (\chi) \frac{H(s)}{s} \, ds,
\end{equation}
where 
\[ H(s) := \prod_{\substack{\alpha \in A \\ \beta \in B}} 
\Big( s^2 -\Big( \frac{\alpha + \beta}{2} \Big)^2 \Big)^3. \]
Expanding the $L$-functions into their Dirichlet series we see that 
(\ref{cis-lambda-icke-noll}) is
\begin{equation*}
\Lambda_{A,B}^{0} (\chi) := \Big( \frac{q}{\pi} 
\Big)^{\delta_{A,B}} \sum_{m,n = 1}^{\infty} 
\frac{\sigma_{-A}(m)\sigma_{-B}(n)}{\sqrt{mn}} \chi(m) \overline{\chi}(n) 
W_{A,B}(mn \pi^3 / q^3),
\end{equation*}
where the generalised sum-of-divisors function is defined by 
\begin{equation*}
\sigma_{s_1,...,s_R}(m) := \sum_{m = m_1 \cdots m_R} 
m_{1}^{s_1} \cdots m_{R}^{s_R}
\end{equation*}
and where
\begin{equation*}
W_{A,B}(\xi) := \frac{1}{2\pi} \int_{(1)} 
G_{A_{s},B_{s}} \xi^{-s} \frac{H(s)}{s} \, ds.
\end{equation*}
For future need, we here note that
\begin{equation} \label{cis-divisor-property-ett}
\sigma_{A_{s}}(n) = \sigma_{A}(n) n^s.
\end{equation}
Also,  
\begin{equation*}
\sum_{g = 0}^{\infty} \frac{\sigma_{-A}(p^g)\sigma_{-B}(p^g)}{p^g} 
= \mathcal{B}_{p}(A,B),
\end{equation*}
which tells us that for $\mathrm{Re}(\alpha), \mathrm{Re}(\beta) > 0$ we have 
\begin{equation} \label{cis-divisor-property-tva}
\sum_{(n,q) = 1} \frac{\sigma_{-A}(n)\sigma_{-B}(n)}{n} 
= \Big(\frac{\mathcal{AZ}}{\mathcal{B}_q}\Big)(A,B).
\end{equation}

From the usual functional equation for Dirichlet $L$-functions it follows that
\begin{equation} \label{FE-Lambda}
\Lambda_{A,B} (\chi) = \Lambda_{-B,-A} (\chi).
\end{equation}
Returning to (\ref{cis-lambda-icke-noll}), we now move the line of integration to $\textrm{Re}(s) = -1$ and, by using (\ref{FE-Lambda}), obtain (recall that 
$q > 1$) that
\begin{equation} \label{cis-AFE}
H(0)\Lambda_{A,B} (\chi)= \Lambda_{A,B}^{0} (\chi) + \Lambda_{-B,-A}^{0} (\chi).
\end{equation}

If $\Phi$ is as in the statement of Theorem \ref{sats1}, then we have 
\begin{align*}
\Lambda_{A,B}^{1} (\chi) &:= \int_{-\infty}^{\infty} \Phi(t) 
\Lambda_{A_{it},B_{-it}}^{0} (\chi) \, dt \\
&=\sum_{m,n = 1}^{\infty} \frac{\sigma_{-A}(m)\sigma_{-B}(n)}{\sqrt{mn}} \chi(m) \overline{\chi}(n) V_{A,B}(m,n;q), 
\end{align*}
where 
\begin{equation} \label{cis-V-def}
V_{A,B}(\xi,\eta;\mu) := \Big(\frac{\mu}{\pi}\Big)^{\delta_{A,B}}
\int_{-\infty}^{\infty} \Phi(t) (\xi/\eta)^{-it} 
W_{A_{it},B_{-it}}(\xi \eta \pi^3/\mu^3) \, dt.
\end{equation}

We will want to study asymptotically
\begin{equation*} 
\mathcal{I}_{A,B} = \mathcal{I}_{A,B}(\Psi, \Phi, Q) := H(0) 
\sum_{q = 1}^{\infty} \sideset{}{^\flat}\sum_{\chi \textrm{(mod } q)} \Psi(q/Q) \int_{-\infty}^{\infty} \Phi(t) 
\Lambda_{A_{it},B_{-it}} (\chi) \, dt,
\end{equation*}
where of course $\Psi$ is a fixed function satisfying the criteria in the statement of Theorem \ref{sats1}. Using (\ref{cis-AFE}) we obtain
\begin{equation} \label{CIS-I-och-Delta-samband}
\mathcal{I}_{A,B} = \Delta_{A,B} + \Delta_{-B,-A},
\end{equation}
where
\begin{align} \label{cis-Delta}
\Delta_{A,B} &:= \sum_{q = 1}^{\infty} 
\sideset{}{^\flat}\sum_{\chi \textrm{(mod } q)} 
\Psi(q/Q) \Lambda_{A,B}^{1} (\chi) \nonumber \\
&=\sum_{q = 1}^{\infty} \sideset{}{^\flat}\sum_{\chi \textrm{(mod } q)} \Psi(q/Q) \sum_{m,n = 1}^{\infty} \frac{\sigma_{-A}(m)\sigma_{-B}(n)}{\sqrt{mn}} \chi(m) \overline{\chi}(n) V_{A,B}(m,n;q).
\end{align}

Let us next study $\Delta_{A,B}.$ There is a diagonal contribution coming from the terms $m = n$, which we call $\mathcal{D}_{A,B}.$ Clearly 
\begin{equation} \label{cis-diagonal-ett} 
\mathcal{D}_{A,B} = 
\sum_{\substack{n,q \\ (n,q)=1}}
\frac{\sigma_{-A}(n)\sigma_{-B}(n)}{n} \phi^{\flat}(q) 
\Psi\Big( \frac{q}{Q} \Big) 
V_{A,B}(n,n;q).
\end{equation} 
Recalling (\ref{cis-V-def}) and also (\ref{cis-divisor-property-ett}) and 
(\ref{cis-divisor-property-tva}), we see that (\ref{cis-diagonal-ett}) equals 
\begin{align} \label{CIS-diag-mellansteg}
&\sum_{q = 1}^{\infty} \Big\{ \phi^{\flat}(q) 
\Psi\Big( \frac{q}{Q} \Big) \\
&\times \int_{-\infty}^{\infty} \Phi(t) \frac{1}{2\pi i} 
\int_{(1)} G_{A_{s+it},B_{s-it}} (q/\pi)^{3s + \delta_{A,B}} 
\Big(\frac{\mathcal{AZ}}{\mathcal{B}_q}\Big)(A_{s},B_{s}) 
\frac{H(s)}{s} \, ds \, dt \Big\}. \nonumber
\end{align}
By moving the line of integration to $\mathrm{Re}(s) = -1/2 + \epsilon$ and noting that the integrand has a pole at $s = 0,$ one may obtain that (\ref{CIS-diag-mellansteg}) equals 
\begin{align*}
&H(0) \sum_{q = 1}^{\infty}  
\Psi\Big( \frac{q}{Q} \Big) \phi^{\flat}(q) 
(q/\pi)^{\delta_{A,B}} \Big(\frac{\mathcal{AZ}}{\mathcal{B}_q}\Big)(A,B) 
\int_{-\infty}^{\infty} \Phi(t) G_{A_{it},B_{-it}} dt 
+ O(Q^{3/2 + \epsilon}) \\
&= H(0) \sum_{q = 1}^{\infty}  
\Psi\Big( \frac{q}{Q} \Big) \phi^{\flat}(q) 
\int_{-\infty}^{\infty} \Phi(t) \mathcal{Q}(A_{it},B_{-it};q) \, dt 
+ O(Q^{3/2 + \epsilon}).
\end{align*}

We shall see that, as one may guess at this point, $\mathcal{D}_{A,B}$ does indeed lead to the term corresponding to $S = T = \varnothing$ in (\ref{q-stor}). Similarly $\mathcal{D}_{-B,-A}$ leads to the term corresponding to $S = A,$ $T = B.$ The other terms in (\ref{q-stor}) arise from the non-diagonal contribution.

Consider once again the expression in (\ref{cis-Delta}). Start by looking at the sum over $\chi.$ 
If $(mn,q) = 1$ then
\begin{equation*}
\sideset{}{^\flat}\sum_{\chi \textrm{(mod } q)} \chi(m) \overline{\chi}(n) = \sideset{}{^*}\sum_{\chi \textrm{(mod } q)} \frac{1 + \chi(-1)}{2} \chi(m) \overline{\chi}(n) = \frac{1}{2} \sum_{\substack{q = dr \\ r | (m \pm n)}} 
\mu(d) \varphi(r),
\end{equation*}
where we by $\pm$ mean 
\begin{equation*} 
\sum_{\pm} = \sum_{+} + \sum_{-}. 
\end{equation*} 
Thus we obtain that 
\begin{equation} \label{cis-Delta-uttryck}
\Delta_{A,B} = \frac{1}{2} \sum_{m,n = 1}^{\infty} \frac{\sigma_{-A}(m)\sigma_{-B}(n)}{\sqrt{mn}} 
\sum_{\substack{d,r \\ (dr,mn) = 1 \\ r | (m \pm n)}} 
\varphi(r) \mu(d) \Psi\Big( \frac{dr}{Q} \Big) 
V_{A,B}(m,n;dr).
\end{equation}

For the terms $m \neq n,$ we divide the sum in (\ref{cis-Delta-uttryck}) into two parts, namely 
$\mathcal{S}_{A,B}$ corresponding to $d > D$ and $\mathcal{G}_{A,B}$ corresponding to $d \leqslant D.$ A suitable choice turns out to be $D = Q^{1/4}.$ We thus have 
\begin{equation*}
\Delta_{A,B} = \mathcal{D}_{A,B} + \mathcal{S}_{A,B} + \mathcal{G}_{A,B}.
\end{equation*}

When investigating $\mathcal{S}_{A,B}$ and $\mathcal{G}_{A,B},$ Conrey, Iwaniec and Soundararajan express the condition $r | (m \pm n)$ (or a similar condition) in terms of characters $\psi \textrm{(mod } r)$ and then single out the contribution coming from the principal character $\psi = \psi_{0}.$ Overall the treatment of the off-diagonal contribution is more complicated than the diagonal one. It was a nice achievement by the three authors named above to be able to handle these terms as well, especially considering the fact that the off-diagonal terms actually give a contribution to the main term in our answer. We will below state the outcome of their investigation. Given $\alpha \in A$ and 
$\beta \in B,$ let us define 
\begin{equation*}
A^{*}(\alpha, \beta) := (A - \{\alpha\}) \cup \{-\beta\}.
\end{equation*}
Then we may explicitly write down what their method yields as follows:
\begin{align} \label{CIS-G-och-S}
&\mathcal{G}_{A,B} + \mathcal{S}_{A,B} = \\
&=H(0) 
\sum_{\substack{\alpha \in A \\ \beta \in B}} \sum_{q = 1}^{\infty} 
\Psi\Big( \frac{q}{Q} \Big) \phi^{\flat}(q)
\int_{-\infty}^{\infty} \Phi(t) 
\mathcal{Q}(A^{*}(\alpha,\beta)_{it},B^{*}(\alpha,\beta)_{-it};q) \, dt \nonumber \\
&+ O(Q^{7/4 + \epsilon}). \nonumber
\end{align}
Of course we may find a similar expression to (\ref{CIS-G-och-S}) for $\mathcal{G}_{-B,-A} + \mathcal{S}_{-B,-A}.$ Remembering (\ref{curly-q-symmetry}), we put these two together with the diagonal contribution and conclude that
\begin{align*}
&\mathcal{D}_{A,B} + \mathcal{S}_{A,B} + \mathcal{G}_{A,B} + 
\mathcal{D}_{-B,-A} + \mathcal{S}_{-B,-A} + \mathcal{G}_{-B,-A} \\
&= H(0) \sum_{q = 1}^{\infty} 
\Psi\Big( \frac{q}{Q} \Big) \phi^{\flat}(q)
\int_{-\infty}^{\infty} \Phi(t) 
\mathcal{Q}_{A_{it},B_{-it}}(q) \, dt 
+ O(Q^{7/4 + \epsilon}).
\end{align*}
By using (\ref{CIS-I-och-Delta-samband}), we establish that 
\begin{align} \label{cis-nastan-klar}
&H(0) \sum_{q = 1}^{\infty} \sideset{}{^\flat}\sum_{\chi \textrm{(mod } q)} \Psi(q/Q) \int_{-\infty}^{\infty} \Phi(t) 
\Lambda_{A_{it},B_{-it}} (\chi) \, dt \nonumber \\
&= H(0) \sum_{q = 1}^{\infty} 
\Psi\Big( \frac{q}{Q} \Big) \phi^{\flat}(q)
\int_{-\infty}^{\infty} \Phi(t) 
\mathcal{Q}_{A_{it},B_{-it}}(q) \, dt 
+ O(Q^{7/4 + \epsilon}).
\end{align}

We would like to divide through by $H(0)$ in (\ref{cis-nastan-klar}), however, there is a conceivable problem when $\alpha + \beta = 0$ for $\alpha \in A$ and $\beta \in B.$ To tackle this difficulty, our first step will be to show that $\mathcal{Q}_{A,B}(q)$ is an analytic function of the shifts. The problem is where any of the relevant $\mathcal{Z}(X,Y)$ has a pole. However, we will now show that these singularities must be removable. To do this, we use Lemma 2.5.1 in the article \cite{CFKRS} by Conrey, Farmer, Keating, Rubinstein and Snaith. In the notation used in the latter we take $k = 3,$ $f(s) = 1/s$ and 
\begin{equation*}
K(\alpha_1, \alpha_2, \alpha_3; \alpha_4, \alpha_5, \alpha_6) 
= \mathcal{Q}(\{\alpha_1,\alpha_2,\alpha_3\}, \{-\alpha_4,-\alpha_5,-\alpha_6\};q). 
\end{equation*}
We find that 
\begin{equation} \label{cfkrs-sambandet}
\mathcal{Q}_{\{\alpha_1,\alpha_2,\alpha_3\},\{-\alpha_4,-\alpha_5,-\alpha_6\}}(q) 
= \sum_{\sigma \in \Xi} K(\alpha_{\sigma(1)},\alpha_{\sigma(2)},\alpha_{\sigma(2)}; \alpha_{\sigma(4)},\alpha_{\sigma(5)},\alpha_{\sigma(6)}),
\end{equation}
where $\Xi$ is the set of $\binom{6}{3}$ permutations of $\sigma \in S_{6}$ 
such that 
$\sigma(1) < \sigma(2) < \sigma(3)$ and $\sigma(4) < \sigma(5) < \sigma(6).$
Now suppose that 
\begin{equation} \label{cfkrs-villkor}
\alpha_i \neq \alpha_j, \textrm{for } i \neq j.
\end{equation} 
Then we may use the following formula from \cite{CFKRS}:
\begin{align} \label{cfkrs-repr}
&\sum_{\sigma \in \Xi} K(\alpha_{\sigma(1)},\alpha_{\sigma(2)},\alpha_{\sigma(3)}; \alpha_{\sigma(4)},\alpha_{\sigma(5)},\alpha_{\sigma(6)}) \nonumber \\
&= \frac{-1}{(3!)^2} \frac{1}{(2\pi i)^6} \oint \cdots \oint 
\frac{K(z_1,z_2,z_3;z_4,z_5,z_6)\Delta(z_1,...,z_6)^2}
{\prod_{i=1}^{6} \prod_{j=1}^{6} (z_i - \alpha_j)} \, dz_1 ... dz_6,
\end{align} 
where
\begin{equation*}
\Delta(z_1,...,z_6) := \displaystyle\prod_{1 \leqslant i < j \leqslant 6} (z_j - z_i),
\end{equation*}
and where one integrates about (small) circles enclosing the $\alpha_j$'s.
By choosing the radii of the circles to be $C_i/\log Q \, (i = 1,...,6)$, for any suitably large constants $C_i,$ we may obtain an upper bound (depending on $Q$) for the RHS of (\ref{cfkrs-repr}). By recalling (\ref{cfkrs-sambandet}), we therefore see that the function $\mathcal{Q}_{\{\alpha_1,\alpha_2,\alpha_3\}, \{-\alpha_4,-\alpha_5,-\alpha_6\}}(q)$ remains bounded whenever (\ref{cfkrs-villkor}) is satisfied. This allows us to conclude\footnote{By Riemann's Extension Theorem --- see for example \cite{GR-complex-an}.} that the possible singularities must be removable. We may conclude that 
$\mathcal{Q}_{A,B}(q)$ is an analytic function of the shifts. Let us for future need remark here that it obviously follows that $\mathcal{Q}_{A_{it},B_{-it}}(q)$ is an analytic function of the shifts.

Now let $\alpha_1, \alpha_2, \alpha_3$ denote the shifts in $A$ and let 
$\alpha_4, \alpha_5, \alpha_6$ denote the shifts in $B$ and let us now return to (\ref{cis-nastan-klar}) by rewriting it as follows: 
\begin{equation} \label{cis-cfkrs}
H(0)l(\alpha_1, \alpha_2, \alpha_3, \alpha_4, \alpha_5, \alpha_6) = 
H(0)r(\alpha_1, \alpha_2, \alpha_3, \alpha_4, \alpha_5, \alpha_6) 
+ O(Q^{7/4 + \epsilon}),
\end{equation}
where we for future convenience recall that $H(0)$ by definition clearly is a function of the shifts.
From uniform convergence of the relevant integrals (given any value of $Q$) one may deduce that 
$l(\alpha_1, \alpha_2, \alpha_3, \alpha_4, \alpha_5, \alpha_6)$ is an analytic function of the shifts. As remarked above, $\mathcal{Q}_{A_{it},B_{-it}}(q)$ is an analytic function of the shifts and hence in particular it is bounded. This implies that 
$r(\alpha_1, \alpha_2, \alpha_3, \alpha_4, \alpha_5, \alpha_6)$ is an analytic function of the shifts. 

Thus $(l-r)(\alpha_1, \alpha_2, \alpha_3, \alpha_4, \alpha_5, \alpha_6)$ is an analytic function of the shifts for $\alpha_i \ll 1/\log Q,$ $i = 1,...,6.$ Now consider any $\alpha_i \leqslant C/\log Q,$ $i = 1,...,6.$ We may apply Cauchy's integral formula to the above function in order to obtain that
\begin{align} \label{cauchy-steg-7}
&(l-r)(\alpha_1, \alpha_2, \alpha_3, \alpha_4, \alpha_5, \alpha_6) 
\nonumber \\
&=\frac{1}{(2\pi i)^6} \int \cdots \int 
\int_{\partial D_{1} \times \cdots \times \partial D_{6}}
\frac{(l-r)(\beta_1, \beta_2, \beta_3, \beta_4, \beta_5, \beta_6)}
{(\beta_{1} - \alpha_{1}) \cdots (\beta_{6} - \alpha_{6})} 
\, d\beta_{1} ... d\beta_{6},
\end{align}
where $D = \prod_{i = 1}^{6} D_{i}$ (i.e.\ $D$ is the polydisc defined as the Cartesian product of the open discs $D_{i}$), where 
\[D_{i} := \{s \in \mathbb{C} : |s - \alpha_{i}| < r_{i}\}, \]
with  
\begin{equation*}  
r_{i} = \frac{2^{i+1}C}{\log Q}. 
\end{equation*}
By this construction we find that 
\begin{equation*}
|\beta_{i} + \beta_{j}| \geqslant 2C/\log Q, \textrm{ for any } i \ne j
\end{equation*} 
and in particular
\begin{equation*}
1/H(0) \ll \log^{54}{Q}.
\end{equation*}
Using this together with (\ref{cis-cfkrs}), we deduce that the numerator in the integrand in (\ref{cauchy-steg-7}) is
\begin{equation*}
\ll Q^{7/4+\epsilon} \log^{54}{Q} \ll Q^{7/4 + \epsilon'}.
\end{equation*}
By estimating the integral in (\ref{cauchy-steg-7}) trivially, we finally obtain that 
\begin{equation*}
l(\alpha_1, \alpha_2, \alpha_3, \alpha_4, \alpha_5, \alpha_6)
= r(\alpha_1, \alpha_2, \alpha_3, \alpha_4, \alpha_5, \alpha_6) 
+ O(Q^{7/4 + \epsilon'}),
\end{equation*}
which completes our sketch of the proof (by ``rewriting" $\epsilon'$ as 
$\epsilon$). \qedsymbol

We now make the following definition here:

\begin{definition} \label{reasonable-def}
Let $N_1,$ $N_2$ and $N_3$ be large positive real numbers. $\Psi(t)$ will be said to be reasonable if it is smooth, is compactly supported in $[1,2]$ and if we also have that $\Psi(t) \leqslant N_1$ and $\Psi'(t) \leqslant N_2$ for $t \in [1,2]$. $\Phi(z)$ will be said to be reasonable if it is an entire function which decays rapidly as $|z| \to \infty$ in any fixed horizontal strip and if $\Phi(t) \leqslant N_3 \exp(-t^2)$ for $t \in \mathbb{R}.$ 
\end{definition}

\begin{remark}
Notice that if $\Psi$ and $\Phi$ are reasonable then they satisfy the conditions in Theorem \ref{sats1}.
\end{remark}

\begin{remark} \label{remark-reasonable}
In this text it will from here on always be assumed that $\Psi$ and $\Phi$ are reasonable.
\end{remark}

\section{Interpretation of Theorem \ref{sats1}} \label{section-kap-tva}

Our first step towards understanding what Theorem \ref{sats1} says will be to study $\mathcal{Q}_{A_{it},B_{-it}}(q).$ In particular, looking at (\ref{q-stor}) and (\ref{curly-q-def}), we see the importance of examining 
\begin{equation} \label{ett-tre}
G_{X_{it},Y_{-it}}, 
\end{equation}
\begin{equation} \label{tva-tre}
\mathcal{B}_{p}(X_{it},Y_{-it}), 
\end{equation}
\begin{equation} \label{tva-och-en-halv-tre}
\mathcal{Z}_{p}^{-1}(X_{it},Y_{-it}) 
\end{equation}
and
\begin{equation} \label{tre-tre}
\mathcal{A}(X_{it},Y_{-it}).
\end{equation}
Clearly (\ref{ett-tre}) depends on $t$ and is equal to 
\begin{equation} \label{gamma-uttryck}
\Big|\Gamma \Big( \frac{1}{4} + \frac{it}{2} \Big)\Big|^6 
\end{equation}
when all the shifts in the sets $X$ and $Y$ are zero.
Let us denote the shifts in $X$ by $\alpha_1, \alpha_2$ and $\alpha_3$ and the shifts in $Y$ by $\alpha_4, \alpha_5$ and $\alpha_6.$ Then by looking at the definition given in (\ref{B-def}) we observe that $\mathcal{B}_{q}(X_{it},Y_{-it})$ is equal to the
$e(0 \cdot \theta)$-coefficient in the following product:
\[ \Big(1 + \frac{e(\theta)}{p^{1/2 + it + \alpha_1}} 
+ \frac{e(2\theta)}{p^{1 + 2it + 2\alpha_1}} +...\Big) \times ... \times 
\Big(1 + \frac{e(-\theta)}{p^{1/2 - it + \alpha_6}} 
+ \frac{e(-2\theta)}{p^{1 - 2it + 2\alpha_6}} +...\Big). \] 
This observation immediately tells us that $\mathcal{B}_{q}(X_{it},Y_{-it})$ is independent of $t.$ Moreover, a straight-forward calculation shows that when all the shifts in the sets $X$ and $Y$ are zero, then (\ref{tva-tre}) equals 
\begin{equation} \label{no-shifts-ett}
\Big( 1 + \frac{4}{p} + \frac{1}{p^2} \Big) 
\Big( 1 - \frac{1}{p} \Big)^{-5}. 
\end{equation}
By going back to (\ref{z-cis-def}) we see that (\ref{tva-och-en-halv-tre}) clearly is independent of $t$ and when all the shifts in the sets $X$ and $Y$ are zero then (\ref{tva-och-en-halv-tre}) equals 
\begin{equation} \label{no-shifts-tva}
\Big( 1 - \frac{1}{p} \Big)^{9}.
\end{equation}
Thus $\mathcal{A}(X_{it},Y_{-it})$ is independent of $t$ and (\ref{tre-tre}) equals
\begin{equation} \label{no-shifts-tre}
\prod_{p} \Big\{ \Big( 1 + \frac{4}{p} + \frac{1}{p^2} \Big) 
\Big( 1 - \frac{1}{p} \Big)^{4} \Big\} \; \; \; (=a_3)
\end{equation}
when all the shifts in the sets $X$ and $Y$ are zero.

Our next goal will be to explore what Theorem \ref{sats1} says when all the shifts are zero. 

\begin{corollary}[Corollary to Theorem 2] \label{cor1}
Letting all the shifts in Theorem \ref{sats1} equal zero, we obtain that 
\begin{align} \label{cor1-ekv}
&\sum_{q \geqslant 1} \; \sideset{}{^\flat}\sum_{\chi \textrm{\emph{(mod} } q)} \Psi\Big(\frac{q}{Q}\Big) \int\limits_{-\infty}^{\infty} 
\Phi(t) |L(1/2+it,\chi)|^6 |\Gamma((1/2+it)/2)|^6 \, dt \nonumber \\
&= a_3 \sum_{q \geqslant 1} \Psi\Big(\frac{q}{Q}\Big) 
\prod_{p|q} \frac{(1-\frac{1}{p})^5}{(1+\frac{4}{p}+\frac{1}{p^2})} 
\phi^\flat(q) \Big(\frac{42\log^9 q}{9!} + O\big(\log^8 q\big)\Big) \nonumber \\ 
&\times \int\limits_{-\infty}^{\infty} \Phi(t) 
|\Gamma((1/2+it)/2)|^6 \, dt \nonumber \\
&= \frac{42 a_3(\mathcal{L})}{9!} Q^2 \log^9 Q 
\int\limits_{0}^{\infty} \frac{\Psi(x) x}{2} \, dx \int\limits_{-\infty}^{\infty} \Phi(t) 
|\Gamma((1/2+it)/2)|^6 \, dt \nonumber \\
&+ O\big(Q^2 \log^8 Q\big).
\end{align}
\end{corollary}

\begin{proof}
We apply Theorem 2 with $A = \{\delta,2\delta,4\delta\}$ and 
$B = \{\delta,2\delta,4\delta\}$ and then let $\delta \to 0.$ The $\mathcal{Q}_{A_{it},B_{-it}}(q)$ in the integrand of Theorem \ref{sats1} is a sum of $20$ terms of type $\mathcal{Q}(X,Y;q).$ If we treat $\delta$ as a complex variable then all the factors of (\ref{curly-q-def}) are analytic functions of $\delta,$ except $\mathcal{Z}(X,Y)$ which is a meromorphic function of $\delta,$ with a pole of order $9$ at $\delta = 0.$ Therefore each of those $20$ terms will have a Laurent series in $\delta.$ However, when we take the sum of them, we must get an analytic function of $\delta$, since the LHS of Theorem \ref{sats1} is continuous with respect to $\delta$ 
(as $\delta \to 0$). As we will let $\delta \to 0,$ we are not interested in the coefficients of the positive $\delta$-powers, but rather want to find the $\delta^0$-coefficient. The key here is to focus on the contribution coming from taking the $\delta^9$-coefficient from the term 
$( \frac{q}{\pi} )^{\delta_{X,Y}},$ since this yields the highest power of $\log q.$ Note that this part of the $\delta^0$-coefficient is equal to $42\log ^9 q/9!.$ This establishes the first equality-statement of this corollary.

Now we will show the second equality-statement of this corollary. Suppose for the moment that we have managed to prove that the main term on the second line of this corollary is equal to the third line. Then by treating the error term in a similar way (one may want to notice that $|\Psi(x)| \ll 1$), we obtain the desired result. 

Thus we are done if we can show that 
\begin{align} \label{cor-steg-som-behovs}
&a_3 \sum_{q \geqslant 1} \Psi\Big(\frac{q}{Q}\Big) 
\prod_{p|q} \frac{(1-\frac{1}{p})^5}{(1+\frac{4}{p}+\frac{1}{p^2})} 
\phi^{\flat}(q) \log^9 q \\
&= a_3(\mathcal{L}) Q^2 \log^9 Q 
\int\limits_{0}^{\infty} \frac{\Psi(x)x}{2}  \, dx + O\big(Q^2 \log^8 Q\big). \nonumber
\end{align}
Let us make two definitions.
\[ G(q):= \frac{\phi^\flat(q) \log^9 q}{q} 
\prod_{p|q} \frac{(1-\frac{1}{p})^5}{(1+\frac{4}{p}+\frac{1}{p^2})} \]
and 
\[ H(P):= \sum_{q \leqslant P} G(q). \]
Assume for the moment that 
\begin{equation} \label{H-egenskap}
H(t) = \frac{a_{3}(\mathcal{L})}{2a_{3}} t \log^9 t + O(t^{13/16}).
\end{equation}
Then by partial summation we have that (\ref{cor-steg-som-behovs}) equals 
\begin{align*} 
&a_3 \sum_{q \geqslant 1} q \Psi\Big(\frac{q}{Q}\Big) G(q) 
= a_3 \sum_{Q < q \leqslant 2Q} q \Psi\Big(\frac{q}{Q}\Big) G(q) \\
&= -a_3 \int\limits_{Q}^{2Q} \frac{d}{dt} \Big(t \Psi\Big(\frac{t}{Q}\Big)\Big) 
H(t) \, dt \\ 
&= -a_3 \int\limits_{Q}^{2Q} \frac{d}{dt} \Big(t \Psi\Big(\frac{t}{Q}\Big)\Big) 
\Big\{ \frac{a_{3}(\mathcal{L})}{2a_{3}} t \log^9 t + O(t^{13/16}) \Big\} 
\, dt \\
&= \frac{a_3(\mathcal{L})}{2} \int\limits_{Q}^{2Q} t \Psi\Big(\frac{t}{Q}\Big) 
\Big\{ \log^9 t + O(\log^8 t) \Big\} \, dt 
+ O\Big(\int\limits_{Q}^{2Q} \Big|\frac{d}{dt} \Big(t \Psi\Big(\frac{t}{Q}\Big)\Big)\Big| t^{13/16} \, dt \Big) \\
&= a_3(\mathcal{L}) Q^2 \log^9 Q 
\int\limits_{0}^{\infty} \frac{\Psi(x)x}{2}  \, dx + O\big(Q^2 \log^8 Q\big).
\end{align*}

To complete this proof we must prove the assumption (\ref{H-egenskap}). 
Trivially 
\begin{equation} \label{1-delen-av-char-bla}
\sum_{q \leqslant P} \frac{\log^9 q}{q} 
\prod_{p|q} \frac{(1-\frac{1}{p})^5}{(1+\frac{4}{p}+\frac{1}{p^2})}
\ll \sum_{q \leqslant P} \frac{\log^9 P}{q} \cdot 1 
\ll \log^{10} P \ll P^{13/16}.
\end{equation}
Therefore, since $\phi^{\flat}(q) = \frac{\phi^{*}(q)}{2} + O(1),$ 
where $\phi^{*}(q)$ is denoting the number of primitive Dirichlet characters modulo $q,$ (\ref{H-egenskap}) will follow if we show that
\begin{equation} \label{H-och-stjarna}
\sum_{q \leqslant t} \frac{\phi^{*}(q) \log^9 q}{q} 
\prod_{p|q} \frac{(1-\frac{1}{p})^5}{(1+\frac{4}{p}+\frac{1}{p^2})} = \frac{a_{3}(\mathcal{L})}{a_{3}} t \log^{9}{t} + O(t \log^{8}{t}).
\end{equation}
Now $\phi^{*}(q)$ is a well-known multiplicative function given by 
$\phi^{*}(p) = p - 2$ for primes $p$ and $\phi^{*}(p^k) = p^{k} (1-1/p)^2$ for 
$k \geqslant 2$.
By standard partial summation, (\ref{H-och-stjarna}) follows from that, with $C$ being some constant, the following result holds: 
\begin{equation} \label{H-och-stjarna-tva}
\sum_{q \leqslant x} \frac{\phi^{*}(q)}{q^2} 
\prod_{p|q} \frac{(1-\frac{1}{p})^5}{(1+\frac{4}{p}+\frac{1}{p^2})} 
= \frac{a_{3}(\mathcal{L})}{a_{3}} \log x + C + O(x^{-3/16}).
\end{equation}

Now we prove (\ref{H-och-stjarna-tva}). Note that we can without loss of generality assume that $x \equiv \frac{1}{2} (\textrm{mod } 1)$.
For Re($s$) $> 1,$ introduce 
\begin{align} 
& K(s):= \sum_{q \geqslant 1} \frac{\phi^{*}(q)}{q^{1+s}} 
\prod_{p|q} \frac{(1-\frac{1}{p})^5}{(1+\frac{4}{p}+\frac{1}{p^2})} \\
&= \prod_{p} \Big(1 + \frac{\phi^{*}(p)}{p^{1+s}} 
\cdot \frac{(1 - \frac{1}{p})^5}{(1 + \frac{4}{p} + \frac{1}{p^2})} 
+ \frac{(1 - \frac{1}{p})^7}{(1 + \frac{4}{p} + \frac{1}{p^2})} 
\cdot \Big(\frac{1}{p^{2s}} + \frac{1}{p^{3s}} + \ldots \Big) \Big) \nonumber \\
&= \prod_{p} \Big\{ (1 - p^{-s})^{-1} \Big\} D(s) \nonumber 
\end{align}
say, where $D(s)$ certainly converges absolutely and uniformly for 
Re$(s) > 1/4$ and 
\begin{equation} \label{res-hjalp}
D(s)|_{s=1} = \frac{a_3(\mathcal{L})}{a_3}.
\end{equation}
Now we use Lemma 3.12 in Titchmarsh \cite{tit}. In the notation used there we take $f(s) = K(s),$ $c$ to be a small fixed positive constant and then also 
$\sigma = 1$ and $T = x$.  
Then we obtain that
\begin{equation}
\sum_{q \leqslant x} \frac{\phi^{*}(q)}{q^2} 
\prod_{p|q} \frac{(1-\frac{1}{p})^5}{(1+\frac{4}{p}+\frac{1}{p^2})}
= \frac{1}{2\pi i} \int\limits_{c - ix}^{c + ix} K(1+s) \frac{x^s}{s} \; ds 
+ O(x^{c-1}).
\end{equation}
Next we move the line of integration to Re$(s) = -1/2$. 
The main term in our answer comes from the residue at $s = 0$. 
To bound the contribution from the two horizontal and the one vertical line of integration we may for example work along the lines of Chapter 12 in \cite{tit}.
The contribution from the horizontal lines is $\ll x^{c-5/6}$. The contribution from the vertical line is certainly $\ll x^{c - 1/2}$.
Thus picking $c = 5/16$ we obtain that 
\begin{align}
\sum_{q \leqslant x} \frac{\phi^{*}(q)}{q^2} 
\prod_{p|q} \frac{(1-\frac{1}{p})^5}{(1+\frac{4}{p}+\frac{1}{p^2})}
&= \textrm{Residue}_{s=0} \Big\{ K(1+s) \frac{x^s}{s} \Big\} + O(x^{-3/16}) \nonumber \\
&= \frac{a_{3}(\mathcal{L})}{a_{3}} \log{x} + C + O(x^{-3/16}),
\end{align}
where we use (\ref{res-hjalp}) when evaluating the residue. 
\end{proof}

\begin{remark} \label{partial-summ-kvall}
As a side-note, we note for future convenience that partial summation applied to (\ref{H-och-stjarna-tva}) implies that 
\begin{equation} \label{H-och-stjarna-tre}
\sum_{q \leqslant t} \frac{\phi^{*}(q)}{q} 
\prod_{p|q} \frac{(1-\frac{1}{p})^5}{(1+\frac{4}{p}+\frac{1}{p^2})} 
= \frac{a_{3}(\mathcal{L})}{a_{3}} t + O(t^{13/16}) 
\end{equation}
and hence that 
\begin{equation} \label{H-och-stjarna-fyra}
H_2(t) := \sum_{q \leqslant t} \frac{\phi^{\flat}(q)}{q} 
\prod_{p|q} \frac{(1-\frac{1}{p})^5}{(1+\frac{4}{p}+\frac{1}{p^2})} 
= \frac{a_{3}(\mathcal{L})}{2a_{3}} t + O(t^{13/16}).
\end{equation}
Let 
\begin{equation}
G_2(q) = \frac{\phi^{\flat}(q)}{q} 
\prod_{p|q} \frac{(1-\frac{1}{p})^5}{(1+\frac{4}{p}+\frac{1}{p^2})}.
\end{equation}
By working as in the proof of Corollary \ref{cor1} we get that 
\begin{align} \label{H-och-stjarna-och-integrering} 
&a_3 \sum_{q \geqslant 1} \Psi\Big(\frac{q}{Q}\Big) \phi^{\flat}(q) 
\prod_{p|q} \frac{(1-\frac{1}{p})^5}{(1+\frac{4}{p}+\frac{1}{p^2})} 
= a_3 \sum_{q \geqslant 1} q \Psi\Big(\frac{q}{Q}\Big) G_2(q) \nonumber \\
&\quad = a_3 \sum_{Q < q \leqslant 2Q} q \Psi\Big(\frac{q}{Q}\Big) G_2(q) 
= - a_3 \int\limits_{Q}^{2Q} \frac{d}{dt} \Big(t \Psi\Big(\frac{t}{Q}\Big)\Big) 
H_2(t) \; dt \nonumber \\
&\qquad = - a_3 \int\limits_{Q}^{2Q} \frac{d}{dt} \Big(t \Psi\Big(\frac{t}{Q}\Big)\Big) 
\Big\{ \frac{a_{3}(\mathcal{L})}{2a_{3}} t + O(t^{13/16}) \Big\} 
\; dt \nonumber \\
&\qquad = \frac{a_3(\mathcal{L})}{2} 
\int\limits_{Q}^{2Q} t \Psi\Big(\frac{t}{Q}\Big) \; dt 
+ O\Big(\int\limits_{Q}^{2Q} \Big|\frac{d}{dt} \Big(t \Psi\Big(\frac{t}{Q}\Big)\Big)\Big| t^{13/16} \; dt \Big) \nonumber \\
&\qquad = a_3(\mathcal{L}) Q^2 
\int\limits_{0}^{\infty} \frac{\Psi(x)x}{2}  \; dx + O\big(Q^{29/16} \big).
\end{align}
\end{remark}

\begin{lemma} \label{trivial-bound}
Suppose that all shifts are zero in Theorem \ref{sats1}. Then the LHS of Theorem \ref{sats1} is $\ll Q^2 \log^9 Q.$
\end{lemma}

\begin{proof}
We use Corollary \ref{cor1} and then it remains to show that 
\[ 42 a_3(\mathcal{L}) Q^2 \frac{\log^9 Q}{9!} 
\int\limits_{0}^{\infty} \frac{\Psi(x) x}{2} \, dx \int\limits_{-\infty}^{\infty} \Phi(t) 
|\Gamma((1/2+it)/2)|^6 \, dt \ll Q^2 \log^9 Q. \] 
Clearly we have that \[ \Big| \int\limits_{0}^{\infty} 
\frac{\Psi(x) x}{2} \, dx\Big| \leqslant 
\int\limits_{1}^{2} \frac{|\Psi(x)| x}{2} \, dx 
\ll \int\limits_{1}^{2} \frac{1 \cdot x}{2} \, dx \ll 1.\]
Also, since $\Gamma(\sigma+it)$ is a bounded function on the line $\sigma = 1/4$, we find that
\[ \Big|\int\limits_{-\infty}^{\infty} \Phi(t) |\Gamma((1/2+it)/2)|^6 \, dt\Big| 
\ll \int\limits_{-\infty}^{\infty} |\Phi(t)| \, dt 
\ll \int\limits_{-\infty}^{\infty} \exp(-t^2) \, dt \ll 1. \] 
This completes the proof.
\end{proof}

\begin{remark} \label{tail-remark}
Notice that the ``tail" of the integral 
\[ \int\limits_{-\infty}^{\infty} \Phi(t) |\Gamma((1/2+it)/2)|^6 \, dt \]
is unimportant. Indeed, by using the following Chernoff bound (see \cite{chernoff}) for the complimentary error function 
\begin{equation} \label{chernoff}
\textrm{erfc}(x) := \frac{2}{\sqrt{\pi}} \int\limits_{x}^{\infty} \exp(-t^2) \, dt \leqslant \exp(-x^2), \; x \geqslant 0
\end{equation} 
we see that
\begin{align} \label{svans}
\int\limits_{\log Q}^{\infty} \Phi(t) |\Gamma((1/2+it)/2)|^6 \, dt 
&\ll \int\limits_{\log Q}^{\infty} \exp(-t^2) \cdot 1 \, dt \\
&\ll \exp(-(\log Q)^2) = Q^{-\log Q} \ll Q^{-1}, \nonumber
\end{align}
and similarly for the integral-part between $-\infty$ and $-\log Q.$ 
In particular we could change the integral involving $\Phi(t)$ in Corollary \ref{cor1} so that it just goes from $-\log Q$ to $\log Q$ instead of from 
$-\infty$ to $\infty.$
\end{remark}

\section{Simplification of Theorem \ref{sats1} to obtain \\ Theorem~\ref{sats1.3}} \label{section-kap-tre}

For future need we will in the next three lemmas basically show that the expressions (\ref{gamma-uttryck})-(\ref{no-shifts-tre}) for (\ref{ett-tre})-(\ref{tre-tre}) respecively remain rather accurate when we allow small shifts.

\begin{lemma} \label{hjalp-lemma-ett}
Suppose that $|A| = |B| = 3$ and that $\alpha, \beta \ll 1/\log Q,$ for $\alpha \in A$, $\beta \in B$. Then 
$G_{A_{it},B_{-it}} = 
|\Gamma((1/2 + it)/2)|^6 
\Big( 1 + O\Big( \frac{\log \log Q}{\log Q} \Big) \Big),$
for $|t| \leqslant \log Q.$
\end{lemma}

\begin{proof}
By definition
\[ G_{A_{it},B_{-it}} = \prod_{\alpha \in A} 
\Gamma \Big( \frac{1}{4} + \frac{it}{2} + \frac{\alpha}{2} \Big) 
\prod_{\beta \in B} 
\Gamma \Big( \frac{1}{4} - \frac{it}{2} + \frac{\beta}{2} \Big). \]
We thus need to show that 
\[ \Gamma \Big( \frac{1}{4} + \frac{it}{2} + \frac{\alpha}{2} \Big)
= \Gamma \Big( \frac{1}{4} + \frac{it}{2} \Big) 
\Big(1 + O\Big( \frac{\log \log Q}{\log Q} \Big)\Big), \]
for $\alpha \ll 1/\log Q$ and $|t| \leqslant \log Q.$
In order to do this we use the Fundamental Theorem of Calculus
\[ F(\gamma(b)) - F(\gamma(a)) = \int\limits_{\gamma} F'(z) \, dz, \]
with $F(z) = \log \Gamma(z).$
We have 
\[ F'(z) = \frac{\Gamma'(z)}{\Gamma(z)} \ll \log (2 + |t|) \ll \log \log Q. \]
Therefore we trivially obtain that
\[ F\Big(\frac{1}{4} + \frac{it}{2} + \frac{\alpha}{2}\Big) 
- F\Big(\frac{1}{4} + \frac{it}{2}\Big) \ll \frac{\log \log Q}{\log Q}. \]
Taking exponentials, we obtain the desired result.
\end{proof}

\begin{lemma} \label{hjalp-lemma-tva}
For $q \ll Q$ and shifts $\alpha, \beta \ll 1/\log Q,$ there exists 
$M \in \mathbb{N}$ such that 
\begin{equation}
\mathcal{B}_q (A_{it},B_{-it}) 
= \prod_{p|q} \Big\{ \Big( 1 + \frac{4}{p} + \frac{1}{p^2} \Big) 
\Big( 1 - \frac{1}{p} \Big)^{-5} \Big\} \cdot 
\Big( 1 + O\Big( \frac{(\log \log Q)^M}{\log Q} \Big) \Big). 
\end{equation} 
\end{lemma}

\begin{proof}
Assume for the moment that we know that
\begin{equation} \label{lemma-tva-i} 
\mathcal{B}_q (A_{it},B_{-it}) \ll (\log \log Q)^M 
\end{equation}
and that 
\begin{equation} \label{lemma-tva-ii} 
\frac{\partial \mathcal{B}_p}{\partial \alpha_i} 
\ll \frac{\log p}{p},
\end{equation}
where in (\ref{lemma-tva-ii}) we have $p$ denoting a prime and where we mean the partial derivative with respect to any of the shifts in $A$ or $B$.  

Then since $\mathcal{B}_q = \prod_{p|q} \mathcal{B}_p,$
we find that
\begin{align} 
\frac{\partial \mathcal{B}_q}{\partial \alpha_i}
= \sum_{p|q} \Big\{ \mathcal{B}_{q/p} \cdot 
\frac{\partial \mathcal{B}_p}{\partial \alpha_i} \Big\} 
&\ll \sum_{p|q} \Big\{ (\log \log Q)^M \cdot 
\frac{\log p}{p} \Big\} \nonumber \\ 
&= (\log \log Q)^M \sum_{p|q} \frac{\log p}{p} 
\ll (\log \log Q)^{M+1}.
\end{align}
Therefore (remember (\ref{no-shifts-ett})), by the Fundamental Theorem of Calculus, we trivially get that
\begin{equation}
\mathcal{B}_q (A_{it},B_{-it}) 
- \prod_{p|q} \Big\{ \Big( 1 + \frac{4}{p} + \frac{1}{p^2} \Big) 
\Big( 1 - \frac{1}{p} \Big)^{-5} \Big\} 
\ll (\log \log Q)^{M+1} (\log Q)^{-1} 
\end{equation} 
and since 
\begin{equation}
\prod_{p|q} \Big\{ \Big( 1 + \frac{4}{p} + \frac{1}{p^2} \Big) 
\Big( 1 - \frac{1}{p} \Big)^{-5} \Big\} \geqslant 1 
\end{equation} 
we get the result (by ``rewriting" $M + 1$ as $M$). 

It thus remains to show (\ref{lemma-tva-i}) and (\ref{lemma-tva-ii}).
We start with the first statement. Suppose that 
\begin{equation} \label{b-p-liten-sak}
\mathcal{B}_p = 1 + O(p^{-1}), \textrm{ for prime } 
p \leqslant q (\ll Q).
\end{equation}
Then
\begin{align} 
&\mathcal{B}_q := \prod_{p|q} \mathcal{B}_p 
\ll \prod_{p|q} (1 + c/p) 
\leqslant \prod_{p|q} \Big(1 + \frac{1+[c]}{p}\Big) 
\leqslant \prod_{p|q} (1 + 1/p)^{1+[c]} = \nonumber \\
&\indent \indent \; = \Big\{ \prod_{p|q} (1 + 1/p)\Big\}^{1+[c]}
\ll (\log \log Q)^{1+[c]}.
\end{align} 
Let us denote the shifts in $A$ by $\alpha_1, \alpha_2$ and $\alpha_3$ and the shifts in $B$ by $\alpha_4, \alpha_5$ and $\alpha_6$. Now in order to prove (\ref{b-p-liten-sak}) we recall that $\mathcal{B}_p$ is the 
$e(0 \cdot \theta)$-coefficient in the following product:
\begin{equation} \label{bp-uttryck-lang-form}
\Big(1 + \frac{e(\theta)}{p^{1/2 + \alpha_1}} 
+ \frac{e(2\theta)}{p^{1 + 2\alpha_1}} + \ldots \Big) \times \ldots \times 
\Big(1 + \frac{e(-\theta)}{p^{1/2 + \alpha_6}} 
+ \frac{e(-2\theta)}{p^{1 + 2\alpha_6}} + \ldots \Big). 
\end{equation} 
It is easy to spot that the answer will begin with 
\begin{equation} \label{borjan-bp}
1 + \Big(\displaystyle\sum_{i=1}^{3} p^{-1/2-\alpha_{i}}\Big)
\Big(\displaystyle\sum_{i=4}^{6} p^{-1/2-\alpha_{i}}\Big). 
\end{equation} 
This is clearly $1 + O(p^{-1})$. Let us finish this part of the proof by showing why the rest of the contribution in (\ref{bp-uttryck-lang-form}) is $O(p^{-2})$.
To do this we consider the worst possible scenario, which clearly is when all the shifts are equal to $-\frac{c}{\log Q}$. 
Let here $k:= \frac{1}{2} - \frac{c}{\log Q}$. We are thus considering 
\begin{equation}
\Big(1 + \frac{e(\theta)}{p^{k}} 
+ \frac{e(2\theta)}{p^{2k}} + \ldots \Big) \times \ldots \times 
\Big(1 + \frac{e(-\theta)}{p^{k}} 
+ \frac{e(-2\theta)}{p^{2k}} + \ldots \Big). 
\end{equation} 
Now the $e(j\theta)/p^{jk}$-coefficient in 
$\Big(1 + \frac{e(\theta)}{p^{k}} 
+ \frac{e(2\theta)}{p^{2k}} + \ldots \Big)^3$
is $1 + 2 +\ldots + (j+1) \ll j^2$ and similarly for the 
$e(-j\theta)/p^{jk}$-coefficient in 
$\Big(1 + \frac{e(-\theta)}{p^{k}} 
+ \frac{e(-2\theta)}{p^{2k}} + \ldots \Big)^3$.
Thus ignoring the contribution mentioned in (\ref{borjan-bp}), the rest is then 
\begin{equation}
\ll \displaystyle\sum_{j=2}^{\infty} \frac{j^4}{p^{2jk}} 
\ll \frac{1}{p^{4k}} = p^{-2 + 4c/\log Q} \ll p^{-2}. 
\end{equation} 

Let us finally show (\ref{lemma-tva-ii}). This is similar to the above discussion. Suppose without loss of generality that the partial differentiation is with respect to $\alpha_{1}$. Again we do the worst case scenario and the $e(-j\theta)/p^{jk}$-coefficient is of course unchanged. 
To get the $e(j\theta)/p^{jk}$-coefficient we consider 
\begin{equation}
\Big(1 + 2\frac{e(\theta)}{p^{k}} + 3\frac{e(2\theta)}{p^{2k}} + \ldots \Big)
\Big(\frac{e(\theta) \log p}{p^{k}} 
+ 2\frac{e(2\theta) \log p}{p^{2k}} + \ldots \Big). 
\end{equation} 
The $e(j\theta)/p^{jk}$-coefficient becomes 
$(\log p)(1 \cdot j + \ldots + j \cdot 1) \ll j^3 \log p$.
Hence the relevant $e(0 \cdot \theta)$-coefficient is 
\begin{equation}
\ll \displaystyle\sum_{j=1}^{\infty} 
\frac{j^2 \cdot j^3 \cdot \log p}{p^{2jk}} 
\ll \frac{\log p}{p^{2k}} \ll \frac{\log p}{p}. 
\end{equation} 
\end{proof}

When we prove the next lemma it is important to keep (\ref{no-shifts-tva}) and (\ref{no-shifts-tre}) in mind.

\begin{lemma} \label{hjalp-lemma-tre}
\[ \mathcal{A} (A_{it},B_{-it}) 
= \prod_{p} \Big\{ \Big( 1 + \frac{4}{p} + \frac{1}{p^2} \Big) 
\Big( 1 - \frac{1}{p} \Big)^{4} \Big\} \cdot 
\Big( 1 + O\Big( \frac{(\log \log Q)^2}{\log Q} \Big) \Big) \]
for shifts $\alpha, \beta \ll 1/\log Q.$
\end{lemma}

\begin{proof}
We have that 
\begin{align*} \label{tredje-lemma-trippel}
&\mathcal{A} (A_{it},B_{-it}) := \\
&= \prod_{p \leqslant \log Q} \mathcal{B}_{p} (A_{it},B_{-it}) \cdot 
\prod_{p \leqslant \log Q} \mathcal{Z}_{p}^{-1} (A_{it},B_{-it}) \cdot 
\prod_{p > \log Q} (\mathcal{B}_p \mathcal{Z}_{p}^{-1}) (A_{it},B_{-it}) \\
&=: F_{1} \cdot F_{2} \cdot F_{3}.
\end{align*} 

By a very similar proof to the proof of Lemma \ref{hjalp-lemma-tva} we see that 
\[ F_{1} = \Big( \prod_{p \leqslant \log Q} 
\Big\{ \Big( 1 + \frac{4}{p} + \frac{1}{p^2} \Big) 
\Big( 1 - \frac{1}{p} \Big)^{-5} \Big\} \Big) \cdot 
\Big( 1 + O\Big( \frac{(\log \log Q)^M}{\log Q} \Big) \Big). \] 

Obviously, since all the shifts are assumed to be $\ll 1/\log Q,$ the same result holds for the sum of any two of the shifts. Therefore, letting $\alpha_{i,j}$ below stand for a complex number $\ll 1/\log Q,$ we will study 
\begin{align} 
\prod_{p \leqslant \log Q} 
\frac{(1 - \frac{1}{p^{1+\alpha_{i,j}}})}{(1 - p^{-1})}
&= \prod_{p \leqslant \log Q} \Big\{ (1 - p^{-1})^{-1}
\Big(1 - \frac{1}{p}\Big\{ 1 + O\Big( \frac{\log \log Q}{\log Q} 
\Big) \Big\} \Big) \Big\} \nonumber \\
&= \prod_{p \leqslant \log Q} 
\Big\{ 1 + O\Big( \frac{\log \log Q}{p \log Q} 
\Big) \Big\}. \label{tredje-l-del-ii-tja} 
\end{align}
Taking logarithms in (\ref{tredje-l-del-ii-tja}) we get
\[ \sum_{p \leqslant \log Q} 
\log {\Big\{ 1 + O\Big( \frac{\log \log Q}{p \log Q} \Big) \Big\}} \]
which is
\[ \ll \Big(\frac{\log \log Q}{\log Q}\Big) 
\sum_{p \leqslant \log Q} \frac{1}{p}
\ll \frac{(\log \log Q)^2}{\log Q}. \]
Taking exponentials this shows that 
\[ F_{2} = \Big( \prod_{p \leqslant \log Q} 
\Big( 1 - \frac{1}{p} \Big)^{9} \Big) \cdot 
\Big( 1 + O\Big( \frac{(\log \log Q)^2}{\log Q} \Big) \Big). \] 

Since we may choose to let all the shifts equal zero in the following result, we will complete the proof of this lemma by showing that
\[ F_{3} = 1 + O\Big( \frac{1}{\log Q} \Big). \]
Recall from (\ref{borjan-bp}) and the discussion thereafter that  
\[ \mathcal{B}_p (A_{it},B_{-it}) 
= 1 + \Big(\displaystyle\sum_{i=1}^{3} p^{-1/2-\alpha_{i}}\Big)
\Big(\displaystyle\sum_{i=4}^{6} p^{-1/2-\alpha_{i}}\Big) 
+ O(p^{-2 + 4c/\log Q}). \]
Working from the definition it is easy to show that
\[ \mathcal{Z}_{p}^{-1} (A_{it},B_{-it}) 
= 1 - \Big(\displaystyle\sum_{i=1}^{3} p^{-1/2-\alpha_{i}}\Big)
\Big(\displaystyle\sum_{i=4}^{6} p^{-1/2-\alpha_{i}}\Big) 
+ O(p^{-2 + 4c/\log Q}). \]
It follows that 
\begin{equation} \label{hjalp-l-tre-referens}
\mathcal{B}_p (A_{it},B_{-it}) \mathcal{Z}_{p}^{-1} (A_{it},B_{-it}) 
= 1 + O(p^{-2 + 4c/\log Q}). 
\end{equation}
At this point it is then easy to show (for example by first taking logarithms and then taking exponentials) that 
\begin{align*} 
F_{3} 
&:= \prod_{p > \log Q} (\mathcal{B}_p \mathcal{Z}_{p}^{-1}) (A_{it},B_{-it}) \\
&= \prod_{p > \log Q} \Big(1 + O\Big(p^{-2 + 4c/\log Q}\Big)\Big) 
= 1 + O\Big( \frac{1}{\log Q} \Big). 
\end{align*}
\end{proof}

In the next theorem, we have replaced the 
$\mathcal{Q}_{A_{it},B_{-it}}(q)$-terms featuring in Theorem \ref{sats1} by a similar expression that is more independent of the shifts and we have also cut off the ``tail" of the integral involving $\Phi(t)$. Basically the result continues to hold, albeit our new error term is bigger.

\begin{definition} \label{P-def}
Define
\begin{equation} 
\mathcal{P}_{A,B}(q)
:= \displaystyle\sum_{\substack{S \subseteq A \\ T \subseteq B \\ |S| = |T|}} 
\mathcal{P}(\bar{S} \cup (-T),\bar{T} \cup (-S);q),
\end{equation}
where 
\begin{equation} \label{curly-P-def}
\mathcal{P}(X,Y;q):= \Big( \frac{q}{\pi} \Big)^{\delta_{X,Y}} 
\mathcal{Z}(X,Y).
\end{equation}
\end{definition}

\begin{theorem} \label{sats1.25}
Suppose that $|A| = |B| = 3$ and that $\alpha, \beta \ll 1/\log Q,$ for $\alpha \in A,$ $\beta \in B$. Suppose also that $\Psi$ and $\Phi$ are reasonable. Then there exists $M \in \mathbb{N}$ such that
\begin{align} \label{th-3-ref} 
&\sum_{q \geqslant 1} \Psi\Big(\frac{q}{Q}\Big) \int\limits_{-\infty}^{\infty} \Phi(t) \sideset{}{^\flat}\sum_{\chi \textrm{\emph{(mod} } q)} \Lambda_{A_{it},B_{-it}}(\chi) \; dt \\ 
&= a_3 \sum_{q \geqslant 1} 
\Psi\Big(\frac{q}{Q}\Big) 
\phi^\flat(q) 
\prod_{p|q} \frac{( 1 - \frac{1}{p} )^{5}}{ ( 1 + \frac{4}{p} + \frac{1}{p^2} )}
\mathcal{P}_{A,B}(q) 
\int\limits_{-\log Q}^{\log Q} 
\Phi(t) |\Gamma((1/2 + it)/2)|^6 \; dt \nonumber \\
&\; \; + O\Big(Q^2 \log^8 Q (\log{\log Q})^M\Big), \nonumber
\end{align}
with $a_3$ as given in (\ref{a3-def}).
\end{theorem}

\begin{proof}

Let us denote the shifts in $A$ by $\alpha_1, \alpha_2$ and $\alpha_3$ and the shifts in $B$ by $\alpha_4, \alpha_5$ and $\alpha_6$.
Let $C$ be a positive constant.
We will first prove the above claim under the extra assumption that  
\begin{equation} \label{assumption-sum-alphas}
|\alpha_{i} \pm \alpha_{j}| \geqslant 2C/\log Q, \textrm{ for any } i \neq j, 
\end{equation} 
where $\pm$ means that the condition holds both with the plus sign and with the minus sign. Also, we are of course still assuming that all the shifts are 
$\ll 1/\log Q$.

We apply Theorem \ref{sats1} to the LHS of this theorem. In the RHS of (\ref{sats1-ekv}) we split the integral involving $\Phi(t)$ as follows:
\begin{equation}
\int_{-\infty}^{\infty} = \int_{-\infty}^{-\log Q} 
+ \int_{-\log Q}^{\log Q} + \int_{\log Q}^{\infty}. 
\end{equation} 
Let us denote the contributions from these three parts by 
$P_1,$ $P_2$ and $P_3$. We will first treat the main term $P_2$. 
Then we will show that the term $P_3$ can be absorbed into the error term of this theorem and the situation is the same with $P_1$.

By putting together Lemmas \ref{hjalp-lemma-ett}, \ref{hjalp-lemma-tva} and \ref{hjalp-lemma-tre} we have that
\begin{align} \label{hjalp-trippel}
&G_{A_{it},B_{-it}}
\Big(\frac{\mathcal{A}}{\mathcal{B}_q}\Big)(A_{it},B_{-it}) \\ \nonumber
&= |\Gamma((1/2 + it)/2)|^6 a_3 
\prod_{p|q} \frac{( 1 - \frac{1}{p})^{5}}
{( 1 + \frac{4}{p} + \frac{1}{p^2})}
\Big( 1 + O\Big( \frac{(\log \log Q)^M}{\log Q} \Big) \Big), \\ \nonumber
\end{align}
for $|t| \leqslant \log Q$.

In the RHS of (\ref{sats1-ekv}), $\mathcal{Q}_{A_{it},B_{-it}}(q)$ is a sum of $20$ terms. In each of the latter we substitute in (\ref{hjalp-trippel}). The part of $P_2$ coming from the contribution coming from picking out the main term in (\ref{hjalp-trippel}) in all of the $20$ cases is equal to the main term of this theorem. Next we show that the contribution from each of the $20$ error terms coming from the substitution of (\ref{hjalp-trippel}) can be absorbed into the error term of this theorem.  

Consider one of the $20$ error-parts described above. Since the shifts are assumed to be at most of order $1/\log Q,$ we clearly have that
\begin{equation} \label{q-pi-bound}
\Big( \frac{q}{\pi} \Big)^{\delta_{X,Y}} \ll 1.
\end{equation} 
Also, by the extra assumption (\ref{assumption-sum-alphas}), we have that 
\begin{equation} \label{z-bound-efter-q}
\mathcal{Z}(X,Y) \ll \log^9{Q},
\end{equation} 
for all relevant sets $X$ and $Y$.
Therefore by using Remark \ref{partial-summ-kvall}, we obtain the required upper bound. 

In order to complete the proof of this theorem under the extra assumption we have to show that $P_3$ is sufficiently small to be absorbed into the error term of this theorem. Of course this follows if we are able to find a suitable upper bound for the contribution of any individual of the $20$ terms that $\mathcal{Q}_{A_{it},B_{-it}}(q)$ constitutes. 

We will again make use of (\ref{q-pi-bound}) and (\ref{z-bound-efter-q}). However, as we can no longer use (\ref{hjalp-trippel}), we will instead simply find an upper bound for 
\begin{equation}
G_{A_{it},B_{-it}}
\Big(\frac{\mathcal{A}}{\mathcal{B}_q}\Big)(A_{it},B_{-it}). 
\end{equation} 
As the gamma function $\Gamma(s)$ is bounded for 
$2/5 \leqslant \mathrm{Re}(s) \leqslant 3/5,$ we have 
\begin{equation}
G_{A_{it},B_{-it}} \ll 1. 
\end{equation} 
We can deduce from Lemma \ref{hjalp-lemma-tva} that  
\begin{equation}
\mathcal{B}_q^{-1}(A_{it},B_{-it}) \ll
\prod_{p|q} \frac{( 1 - \frac{1}{p})^{5}}
{( 1 + \frac{4}{p} + \frac{1}{p^2})} 
\end{equation} 
and 
Lemma \ref{hjalp-lemma-tre} yields that 
\begin{equation}
\mathcal{A} (A_{it},B_{-it}) \ll 1. 
\end{equation} 
Using all these upper bounds we reach our conclusion upon remembering 
Remarks \ref{partial-summ-kvall} and \ref{tail-remark}. 

Suppose that the theorem is to be proved for all 
$|\alpha_{i}| \leqslant C/\log Q$. Consider now (without the extra assumption) any $|\alpha_{i}| \leqslant C/\log Q$. 
The idea is to use Cauchy's integral formula in order to go from the above ``easier" case to the general case. Certainly the LHS of this theorem is an analytic function of the complex variables $\alpha_1,$ $\alpha_2,$ $\alpha_3,$ $\alpha_4,$ $\alpha_5$ and $\alpha_6,$ let us denote it by 
L$(\alpha_1, \alpha_2, \alpha_3, \alpha_4, \alpha_5, \alpha_6)$. We can therefore replace it by a 6-fold integral involving 
L$(\beta_1, \beta_2, \beta_3, \beta_4, \beta_5, \beta_6)$. Explicitly, let $D$ be the polydisc defined as the Cartesian product of the open discs $D_{i},$ i.e.\ $D = \prod_{i = 1}^{6} D_{i},$ where 
\begin{equation}
D_{i} := \{s \in \mathbb{C} : |s - \alpha_{i}| < r_{i}\}, 
\end{equation} 
with  
\begin{equation}  
r_{i} = \frac{2^{i+1}C}{\log Q}. 
\end{equation}
An application of Cauchy's integral formula yields that
\begin{align} \label{cauchy-steg-6}
&\textrm{L}(\alpha_1, \alpha_2, \alpha_3, \alpha_4, \alpha_5, \alpha_6) 
\nonumber \\
&=\frac{1}{(2\pi i)^6} \int \cdots \int 
\int_{\partial D_{1} \times \cdots \times \partial D_{6}}
\frac{\textrm{L}(\beta_1, \beta_2, \beta_3, \beta_4, \beta_5, \beta_6)}
{(\beta_{1} - \alpha_{1}) \cdots (\beta_{6} - \alpha_{6})} 
\; d\beta_{1} \ldots d\beta_{6}.
\end{align}
Now we notice that the $\beta_i$ satisfy $|\beta_i - \alpha_i| = r_i$. But this is easily seen to imply that $\beta_i \ll 1/\log Q$ and that 
$|\beta_i \pm \beta_j| \geqslant 2C/\log Q,$ so that this theorem applies if the $\beta_i$-terms are seen as shifts. We substitute in the outcome for 
$\textrm{L}(\beta_1, \beta_2, \beta_3, \beta_4, \beta_5, \beta_6)$ into 
(\ref{cauchy-steg-6}). The error term we get in the integrand goes through and stays of course of the same order. The main term that we now have in the numerator of our integrand, let us denote it by 
R$(\beta_1, \beta_2, \beta_3, \beta_4, \beta_5, \beta_6),$ is an analytic 
function of the complex variables $\beta_1, \beta_2, \beta_3, \beta_4, \beta_5$ and $\beta_6,$ which means that we can undo the Cauchy-process and retrieve R$(\alpha_1, \alpha_2, \alpha_3, \alpha_4, \alpha_5, \alpha_6)$. The former can be seen by using Lemma 2.5.1 in the article \cite{CFKRS} 
(by working in a similar way as in the paragraph after (\ref{cis-nastan-klar})). This completes the proof.
\end{proof}

Theorem \ref{sats1.25} is simpler than its ``earlier" version Theorem \ref{sats1} in the sense that some terms in the former are independent of the shifts. We will next find a simpler version of Theorem \ref{sats1.25}. 

\begin{definition} \label{R-def}
Define
\begin{equation} 
\mathcal{R}_{A,B}(Q)
:= \displaystyle\sum_{\substack{S \subseteq A \\ T \subseteq B \\ |S| = |T|}} 
\mathcal{R}(\bar{S} \cup (-T),\bar{T} \cup (-S);Q),
\end{equation}
where 
\begin{equation} 
\mathcal{R}(X,Y;Q):= Q^{\delta_{X,Y}}  
\prod_{\substack{x \in X \\ y \in Y}} 
\frac{1}{x + y}.
\end{equation}
\end{definition}

\begin{theorem} \label{sats1.3}
Suppose that $|A| = |B| = 3$ and that $\alpha, \beta \ll 1/\log Q,$ for 
$\alpha \in A,$ $\beta \in B$. Suppose also that $\Psi$ and $\Phi$ are reasonable. Then there exists $M \in \mathbb{N}$ such that 
\begin{align} \label{sats1.3-ekv}
&\sum_{q \geqslant 1} \Psi\Big(\frac{q}{Q}\Big) \int\limits_{-\infty}^{\infty} \Phi(t) \sideset{}{^\flat}\sum_{\chi \textrm{\emph{(mod} } q)} \Lambda_{A_{it},B_{-it}}(\chi) \; dt \\
&= a_3(\mathcal{L}) w(\Phi) 
\int\limits_{0}^{\infty} \frac{\Psi(x)x}{2} \; dx \; Q^{2} \mathcal{R}_{A,B}(Q) + O\Big(Q^2 \log^8 Q (\log{\log Q})^M\Big), \nonumber
\end{align}
where we have introduced the notation
\begin{equation}
w(\Phi) := 
\int\limits_{-\infty}^{\infty} \Phi(t) |\Gamma((1/2 + it)/2)|^6 \; dt.
\end{equation}
\end{theorem}

\begin{proof}
We will below prove this theorem under the extra assumption (\ref{assumption-sum-alphas}). The general case then follows exactly as in the proof of Theorem \ref{sats1.25}. 

Let us focus on $\mathcal{P}_{A,B}(q)$ in the RHS of Theorem \ref{sats1.25} which is a sum of $20$ terms. We will next look at what such a term looks like (recall (\ref{curly-P-def})).

Now $\mathcal{Z}(X,Y)$ is a product of nine factors of the form 
$\zeta(1+\lambda),$ where $\lambda$ is a linear combination of the shifts (naturally we here do not mean that all $\lambda$ are exactly the same). 
For each factor we may write 
\begin{equation} \label{laurent-zeta-detail}
\zeta(1+\lambda) = \frac{1}{\lambda} + O(1). 
\end{equation}
Notice that due to our assumption (\ref{assumption-sum-alphas}), we have that 
\begin{equation}
1/\lambda \ll \log{Q}. 
\end{equation}
Suppose we multiply together the nine factors of $\mathcal{Z}(X,Y)$. We may choose the term $1/\lambda$ or the term $O(1)$ in (\ref{laurent-zeta-detail}). If we do not in each case choose the former, the corresponding product would be 
\begin{equation}
\ll \log^8 Q. 
\end{equation} 
If we then multiply this by $(\frac{q}{\pi})^{\delta_{X,Y}}$ we would, since (\ref{q-pi-bound}) holds, still obtain something that is 
\begin{equation}
\ll \log^8 Q. 
\end{equation} 
As we have seen before (i.e.\ via use of the fact that (\ref{H-och-stjarna-och-integrering}) $\ll Q^{2}$), one may then conclude that this will give a total contribution of order 
\begin{equation}
\ll Q^2 \log^8 Q 
\end{equation} 
which thus may be absorbed into the error term of this theorem.
Thus (under the assumption of (\ref{assumption-sum-alphas})) we may replace $\mathcal{Z}(X,Y)$ (in Theorem \ref{sats1.25}) by 
\begin{equation}
\displaystyle\prod_{\substack{x \in X \\ y \in Y}} \frac{1}{x + y}. 
\end{equation} 

Next, we consider 
\begin{equation}
\Big(\frac{q}{\pi}\Big)^{\delta_{X,Y}} 
= Q^{\delta_{X,Y}} 
\cdot \exp(\delta_{X,Y} \log(q/Q\pi)). 
\end{equation} 
Notice that since $\Psi$ is compactly supported in $[1,2]$ we may assume that 
$Q \leqslant q \leqslant 2Q.$ Also, clearly $\delta_{X,Y} \ll 1/\log{Q}.$ 
Therefore 
\begin{equation}
\Big(\frac{q}{\pi}\Big)^{\delta_{X,Y}} 
= Q^{\delta_{X,Y}} 
\cdot \Big(1 + O\Big(\frac{1}{\log Q}\Big)\Big). 
\end{equation} 
Since 
\begin{equation}
Q^{\delta_{X,Y}} \ll 1, 
\end{equation} 
the same reasoning as in the previous paragraph tells us that we may replace 
$(\frac{q}{\pi})^{\delta_{X,Y}}$ by $Q^{\delta_{X,Y}}$ (again under the assumption (\ref{assumption-sum-alphas})).

At this point our main term looks like 
\begin{equation} \label{main-termmm-forfining-i}
a_3 \mathcal{R}_{A,B}(Q) 
\int\limits_{-\log Q}^{\log{Q}} \Phi(t) |\Gamma((1/2 + it)/2)|^6 \; dt 
\sum_{q \geqslant 1} \Psi\Big(\frac{q}{Q}\Big) 
\phi^\flat(q) \prod_{p|q} \frac{( 1 - \frac{1}{p} )^{5}}{ 
( 1 + \frac{4}{p} + \frac{1}{p^2} )}. 
\end{equation} 
Using (\ref{H-och-stjarna-och-integrering}) and (\ref{svans}), it follows immediately that (\ref{main-termmm-forfining-i}) is 
\begin{align} \label{main-termmm-forfining-ii} 
&a_3(\mathcal{L}) \times \mathcal{R}_{A,B}(Q) \times 
\Big\{ \int\limits_{-\infty}^{\infty} 
\Phi(t) |\Gamma((1/2 + it)/2)|^6 \; dt + O(Q^{-1}) \Big\} \nonumber \\ 
&\times \Big\{ Q^2 \int\limits_{0}^{\infty} 
\frac{\Psi(x)x}{2}  \; dx + O\big(Q^{29/16} \big) \Big\}.
\end{align}
Recalling that our assumption implies that $\mathcal{R}_{A,B}(Q) \ll \log^9{Q},$ it is easily seen that all but the main term in (\ref{main-termmm-forfining-ii}) can be relegated to the error term of this theorem, leaving us with the desired main term. 
\end{proof}

\section{Differentiated version of Theorem \ref{sats1.3}} \label{Diric-kapitel-differentiating-in-cis}

\begin{theorem} \label{sats1.3-diff}
Suppose that $|A| = |B| = 3$ and that $\alpha, \beta \ll 1/\log Q,$ for 
$\alpha \in A,$ $\beta \in B$. Suppose also that $\Psi$ and $\Phi$ are reasonable. Denote the shifts in $A$ by $\alpha_1, \alpha_2$ and $\alpha_3$ and the shifts in $B$ by $\alpha_4, \alpha_5$ and $\alpha_6$. Then there exists 
$M \in \mathbb{N}$ such that 
\begin{align} 
&\sum_{q \geqslant 1} \Psi\Big(\frac{q}{Q}\Big) \int\limits_{-\infty}^{\infty} \Phi(t) \sideset{}{^\flat}\sum_{\chi \textrm{\emph{(mod} } q)} \frac{\partial^2}{\partial\alpha_i\partial\alpha_j}
\Lambda_{A_{it},B_{-it}}(\chi) \; dt \\
&= a_3(\mathcal{L}) w(\Phi) 
\int\limits_{0}^{\infty} \frac{\Psi(x)x}{2} \; dx \; Q^{2} \frac{\partial^2}{\partial\alpha_i\partial\alpha_j} 
\mathcal{R}_{A,B}(q) + O\Big(Q^2 \log^{10}{Q} (\log{\log Q})^M\Big). \nonumber
\end{align}
\end{theorem}

\begin{proof}
This is very simple. We use ``Cauchy's integral trick" and the result follows immediately. 

To expand slightly on what this means, we begin with Theorem \ref{sats1.3}. Note that the LHS and the main term in the RHS of (\ref{sats1.3-ekv}) are analytic functions of the complex variables $\alpha_1$, $\alpha_2$, $\alpha_3$, $\alpha_4$, $\alpha_5$ and $\alpha_6.$ Take the first one and subtract the latter and we get an analytic function, let us call it 
$D(\alpha_1, \alpha_2, \alpha_3, \alpha_4, \alpha_5, \alpha_6).$ Then Theorem~\ref{sats1.3} tells us that 
\[ D(\alpha_1, \alpha_2, \alpha_3, \alpha_4, \alpha_5, \alpha_6) 
\ll Q^2 \log^{8} Q (\log{\log Q})^M. \] 
Now say that the differentiation is with respect to for example $\alpha_3$ and $\alpha_4.$ Using Cauchy's integral formula with $r = 1/\log Q$ we get 
\begin{align*} 
&\frac{\partial^{2} 
D(\alpha_1, \alpha_2, \alpha_3, 
\alpha_4, \alpha_5, \alpha_6)}{\partial \alpha_3 \partial \alpha_4} \\
&= \frac{1}{(2\pi i)^2} \int_{|w_3 - \alpha_3| = r} \int_{|w_4 - \alpha_4| = r}
\frac{D(\alpha_1, \alpha_2, w_3, w_4, \alpha_5, \alpha_6)}
{(w_3 - \alpha_{3})^2 (w_4 - \alpha_{4})^2} \, dw_3 \, dw_4. 
\end{align*}
Then the pathlengths are each of order $(\log Q)^{-1}$ and the integrand is trivially 
\[ \ll (\log Q)^4 \cdot Q^2 \log^{8} Q (\log{\log Q})^M 
= Q^2 \log^{12} Q (\log{\log Q})^M, \] 
giving 
\[ \frac{\partial^2 D}{\partial\alpha_3 \partial\alpha_4} \ll Q^2 \log^{10} Q (\log{\log Q})^M. \]
Then we remember that clearly the derivative of the difference of two analytic functions equals the difference of the derivatives of those two functions. 
\end{proof}

\section{Building-stones in the proof of Theorem \ref{huvudsats}} \label{section-kap-fem}

\subsection{Main Assumption} 

We will now make an ``assumption". 
\\
\\
\textbf{Main Assumption}: Suppose that for all $\frac{5Q}{4} < q \leqslant \frac{7Q}{4}$ and all even primitive Dirichlet characters 
$\chi \textrm{(mod } q)$ we have that all the gaps with $t \in [T,2T]$ between consecutive zeros of 
$L(\frac{1}{2} + it,\chi)$ are\footnote{For either of the two zeros near the endpoints of the interval we will here mean the distance from them to the respective endpoint.} at most $\frac{3\kappa}{\log Q}.$ 

\begin{definition}[Remark on Main Assumption] \label{kappa-def}
We will in fact take $\kappa$ to be the smallest real number such that the Main Assumption is satisfied. 
\end{definition}

\begin{remark}
Thus $\kappa$ will depend on $Q$ and will obviously be well-defined. 
\end{remark}

\subsection{Introducing the function $f(t,\chi,\kappa)$}

Let $\chi$ be an even primitive Dirichlet character modulo $q$. 
Define 
\[ G(n,\chi) := \displaystyle\sum_{l=1}^q \chi(l)e^{2\pi iln/q}.\] 
Then we have $G(n,\chi) = \overline{\chi(n)}G(1,\chi)$, 
where of course 
\[ G(1,\chi) = \displaystyle\sum_{l=1}^q \chi(l)e^{2\pi il/q}. \]
Also (remember that $\chi$ is even), $G(1,\chi)G(1,\bar{\chi}) = q$ and $|G(1,\chi)|^{2} = q$, i.e.\ $|G(1,\chi)| = \sqrt{q}$.
The well-known function 
\[ \xi(s,\chi) 
:= (q/\pi)^{s/2}\Gamma(s/2)L(s,\chi) \] 
satisfies the functional equation (see \cite[Ch. 9]{dav})
\begin{equation} \label{FE-ett}
\xi(1-s,\bar{\chi}) = K(\chi)\xi(s,\chi),
\end{equation} where 
\begin{equation*} 
K(\chi) := q^{1/2}/G(1,\chi).
\end{equation*}

\begin{definition} \label{U-def}
Let $\chi$ be an even primitive Dirichlet character. Then we define 
\begin{equation*} 
U(s,\chi) := K(\chi)^{1/2} 
\xi(s,\chi), 
\end{equation*}
where\footnote{If $K(\chi) = -1$ then $\chi$ cannot be real and 
in this case we define $\sqrt{K(\chi)}$ to be $i$ or $-i$ 
corresponding to whether or not the imaginary part of the 
first non-real number in the sequence 
$\chi(1), \chi(2),...$ is positive.} 
we choose to define $\log{z} = \log{|z|} + i\arg{z}$, with 
$-\pi < \arg{z} < \pi.$
\end{definition}

Let us first show that 
\begin{equation} \label{U-konj}
U(s,\chi) = \overline{U(\bar{s},\bar{\chi})}.
\end{equation} 
To do this, we first notice that both sides of (\ref{U-konj}) represent 
analytic (or meromorphic in the case $q$ = 1) functions, so that
the identity theorem tells us that we only need to establish
the result for $s = \sigma > 1$. We have then 
\begin{align*}
U(\sigma,\chi) &= \sqrt{K(\chi)}\xi(\sigma,\chi) = 
\sqrt{K(\chi)} \; \overline{\xi(\bar{\sigma},\bar{\chi})} \\
&= \Big(\sqrt{K(\chi)}/\overline{\sqrt{K(\bar{\chi})}}\Big) \overline{\sqrt{K(\bar{\chi})}\xi(\bar{\sigma},\bar{\chi})} \\
&= \Big(\frac{\sqrt{K(\chi)} 
\sqrt{K(\bar{\chi})}}{|K(\bar{\chi})|}\Big) \overline{U(\bar{\sigma},\bar{\chi})}
= \frac{1}{1} \cdot \overline{U(\bar{\sigma},\bar{\chi})} 
= \overline{U(\bar{\sigma},\bar{\chi})}.
\end{align*} 
This property of $U(s,\chi)$ tells us that 
$U(\frac{1}{2} + it,\chi) = \overline{U(\frac{1}{2} - it,\bar{\chi})}$.
On the other hand, from the functional 
equation (\ref{FE-ett}) we have 
\begin{align*}
\overline{U(1/2 - it,\bar{\chi})} 
&= \overline{\sqrt{K(\bar{\chi})}\xi(1/2 - it,\bar{\chi})}
= \overline{\sqrt{K(\bar{\chi})}K(\chi)\xi(1/2 + it,\chi)} \\ 
&= \overline{\sqrt{K(\chi)}\xi(1/2 + it,\chi)}
= \overline{U(1/2 + it,\chi)}.
\end{align*}
Hence $U(1/2 + it,\chi)$ is real (for $t \in \mathbb{R}$). 

\begin{definition} \label{W-def}
For $\chi$ an even primitive Dirichlet character, define
\begin{align*}
W(t,\chi) &:= \Big( \frac{q}{\pi} \Big)^{- \frac{1}{4}} 
U(1/2 + it,\chi) \\
&= K(\chi)^{1/2} 
\Big( \frac{q}{\pi} 
\Big)^{\frac{1}{2}(\frac{1}{2} + it) - \frac{1}{4}} 
\Gamma\Big( \frac{(\frac{1}{2} + it)}{2} \Big) 
L(\tfrac{1}{2} + it, \chi). 
\end{align*} 
\end{definition}

\begin{remark} \label{reality-of-W}
By the discussion above it follows immediately that $W(t,\chi) \in \mathbb{R}$ for $t \in \mathbb{R}.$
\end{remark}

\begin{remark} \label{L-W-0}
It is obvious that $L(\frac{1}{2} + it, \chi) = 0$ precisely when 
$W(t,\chi) = 0.$ 
\end{remark}

We next introduce one more function.

\begin{definition} \label{f-def}
\[ f(t,\chi,\kappa) = f(t,\chi,\kappa,Q) 
:= W(t - \frac{\kappa}{\log Q}, \chi) W(t,\chi) 
W(t + \frac{\kappa}{\log Q}, \chi). \]
\end{definition}

\begin{remark} \label{gap-gaps}
Denote the zeros of $f(t,\chi,\kappa)$ with 
$T + \frac{\kappa}{\log Q} \leqslant t 
\leqslant 2T - \frac{\kappa}{\log Q}$ by 
$t_{1,\chi}, t_{2,\chi},...,t_{N_{\chi},\chi},$ ordered in non-decreasing order. Then keeping Remark \ref{L-W-0} in mind, it is simple\footnote{The reader can, for example by drawing a simple diagram, easily do an argument by contradiction to reach the desired conclusion.} to deduce that the Main Assumption implies that \begin{equation} \label{gap-gaps-ekv}
t_{i+1,\chi} - t_{i,\chi} 
\leqslant \frac{\kappa}{\log Q}, 
\end{equation}
for $i = 1,2,...,N_{\chi}-1.$
\end{remark}

\begin{remark} \label{andpunkter}
For future reference we note here that the Main Assumption implies that 
\[ t_{1,\chi} \leqslant T + \frac{2\kappa}{\log Q} \textrm{ and } 
t_{N_{\chi},\chi} \geqslant 2T - \frac{2\kappa}{\log Q}. \]
\end{remark}

\subsection{Wirtinger's inequality and an application of it}

We begin with the statement of the simplest version of Wirtinger's inequality. 

\begin{theorem} \label{wirtinger-ineq}
Suppose that $y(t)$ is a continuously differentiable function which satisfies 
$y(0) = y(\pi) = 0.$ Then 
\[ \int\limits_{0}^{\pi} y(t)^2 \, dt \leqslant \int\limits_{0}^{\pi} y'(t)^2 
\, dt. \]
\end{theorem}

\begin{proof}
The reader is referred to Theorem 256 in Hardy, Littlewood and P\'{o}lya's \cite{HLP}.
\end{proof}

\begin{corollary} \label{wirtinger-ineq-scaled}
For $i = 1,2,...,N_{\chi}-1$ we have that 
\begin{align*} 
\int\limits_{t_{i,\chi}}^{t_{i+1,\chi}} f(t,\chi,\kappa)^2 \, dt 
&\leqslant \Big( \frac{t_{i+1,\chi} - t_{i,\chi}}{\pi} \Big)^2
\int\limits_{t_{i,\chi}}^{t_{i+1,\chi}} f'(t,\chi,\kappa)^2 \, dt \\
&\leqslant \Big( \frac{\kappa}{\pi \log Q} \Big)^2
\int\limits_{t_{i,\chi}}^{t_{i+1,\chi}} f'(t,\chi,\kappa)^2 \, dt. 
\end{align*}
\end{corollary}

\begin{proof}
One may make a linear substitution in Theorem \ref{wirtinger-ineq} to obtain a similar result if the function $y(t)$ has zeros at two general points $a$ and $b.$ We do so for the function $f(t,\chi,\kappa),$ which is continuously differentiable, and this gives us the first inequality. The latter inequality follows immediately from (\ref{gap-gaps-ekv}).
\end{proof}

\begin{corollary} \label{wirtinger-ineq-scaled-sum}
For any $\frac{5Q}{4} < q \leqslant \frac{7Q}{4}$ and all even primitive Dirichlet characters $\chi (\textrm{\emph{mod} } q)$ we have that 
\[ \int\limits_{t_{1,\chi}}^{t_{N_{\chi},\chi}} f(t,\chi,\kappa)^2 \, dt 
\leqslant \Big( \frac{\kappa}{\pi \log Q} \Big)^2 
\int\limits_{t_{1,\chi}}^{t_{N_{\chi},\chi}} f'(t,\chi,\kappa)^2 \, dt. \]
\end{corollary}

\begin{proof}
We sum up the inequalities in Corollary \ref{wirtinger-ineq-scaled} for $i = 1,2,...,N_{\chi}-~1.$ 
\end{proof}

\begin{corollary} \label{wirtinger-ineq-scaled-sum-sum}
\begin{align*} 
&\displaystyle\sum_{\frac{5Q}{4} < q \leqslant \frac{7Q}{4}} \; \;
\sideset{}{^\flat}\sum_{\chi \textrm{\emph{(mod} } q)} \;
\int\limits_{t_{1,\chi}}^{t_{N_{\chi},\chi}} f(t,\chi,\kappa)^2 \, dt \\
&\leqslant \Big( \frac{\kappa}{\pi \log Q} \Big)^2 
\displaystyle\sum_{\frac{5Q}{4} < q \leqslant \frac{7Q}{4}} \; \;
\sideset{}{^\flat}\sum_{\chi \textrm{\emph{(mod} } q)} \;
\int\limits_{t_{1,\chi}}^{t_{N_{\chi},\chi}} f'(t,\chi,\kappa)^2 \, dt. 
\end{align*}
\end{corollary}

\begin{proof}
We simply sum up the inequalities in Corollary \ref{wirtinger-ineq-scaled-sum} over all the relevant Dirichlet characters.
\end{proof}

\section{Investigating Corollary \ref{wirtinger-ineq-scaled-sum-sum}} \label{Diric-kapitel-investigation}

Our goal in this section is basically to find estimates for the LHS and the RHS in Corollary \ref{wirtinger-ineq-scaled-sum-sum}. This shall then give us an inequality in terms of $\kappa$ which we will consider in Section \ref{Diric-kapitel-conclusion}.

\subsection{Lower bound for the LHS of Corollary \ref{wirtinger-ineq-scaled-sum-sum}} \label{Diric-kapitel-investigation-LHS}

Let us begin by finding a nice expression for a lower bound for the LHS in Corollary \ref{wirtinger-ineq-scaled-sum-sum}. By Remark \ref{andpunkter}, the integration limits are very close to $T$ and $2T$ respectively. Let $\epsilon$ be a small (fixed) positive number. Then clearly the LHS of Corollary \ref{wirtinger-ineq-scaled-sum-sum} is
\begin{equation} \label{olikhets-lhs-delsteg-ett}
\geqslant \displaystyle\sum_{q \geqslant 1} \; 
\chi_{(\frac{5}{4},\frac{7}{4}]} \Big(\frac{q}{Q}\Big) 
\sideset{}{^\flat}\sum_{\chi \textrm{(mod } q)} \;
\int\limits_{T + \epsilon}^{2T - \epsilon} 
f(t,\chi,\kappa)^2 \; dt. 
\end{equation}
Next we notice that if we let the shifts in $A$ and $B$ be 
$\frac{\kappa i}{\log Q},$ $0$ and $-\frac{\kappa i}{\log Q}$ 
(and we do that!), then 
\begin{equation}
f(t,\chi,\kappa)^2 = \Lambda_{A_{it},B_{-it}}(\chi). 
\end{equation} 
Substituting this into (\ref{olikhets-lhs-delsteg-ett}) we have an expression that looks a lot like the LHS of Theorem \ref{sats1}. However, in order to proceed we will need to approximate the characteristic functions by suitable weight-functions $\Psi$ and $\Phi$ that satisfy the conditions of Theorem \ref{sats1.3}. Below $u$ will denote a suitably small (fixed) positive number.

Construction of $\Psi_{1}:$
First we will define
\begin{equation} \label{g-def}
g(x):= \left\{ \begin{array}{l l}
  \exp(-u^2/x) & \quad \mbox{if $x > 0,$}\\
  0 & \quad \mbox{if $x \leqslant 0$.}\\
\end{array} \right. 
\end{equation} 
Then $g(x)$ is smooth and approximates $\chi_{[0,\infty)}(t)$. Now let 
\begin{equation}
\Psi_{1}(t):= g(t-5/4)g(7/4-t). 
\end{equation} 
Then clearly 
\begin{equation} 
\Psi_{1}(t) \leqslant \chi_{(5/4,7/4]}(t). 
\end{equation}

Construction of $\Phi_{1}:$
We define 
\begin{equation}
\Phi_{1}(z) := \frac{1}{\sqrt{\pi} u} 
\int\limits_{T + \epsilon +\sqrt{u}}^{2T - \epsilon - \sqrt{u}} 
\exp(-(v-z)^2/u^2) \; dv 
- \exp\Big((3T)^2 - \frac{1}{u}\Big) \exp(-z^2). 
\end{equation} 
Then $\Phi_{1}(z)$ is reasonable and is pretty ``close" to 
$\chi_{[T,2T]}(t),$ if we see it as a function from $\mathbb{R}$ to $\mathbb{R}$.
Indeed for $t \in \mathbb{R}$ we have the following:  
\begin{equation} \label{phi-ett-olikhet}
\Phi_{1}(t) \leqslant \chi_{[T + \epsilon,2T - \epsilon]}(t). 
\end{equation}
One possible way to show (\ref{phi-ett-olikhet}) is to simply consider what happens when 
$t \in [T + \epsilon,2T - \epsilon],$ when this fails but $t$ is somewhat close to that interval (say $t \in [0,3T]$) and finally when $t \notin [0,3T],$ keeping in mind that the Gaussian integral satisfies 
\begin{equation}
\int\limits_{-\infty}^{\infty} \exp(-t^2) \; dt = \sqrt{\pi} 
\end{equation} 
and the Chernoff bound (\ref{chernoff}) for the complimentary error function. Explicitly, in the first of the three above cases one trivially has 
\begin{equation}
\Phi_{1}(t) \leqslant \frac{2}{\sqrt{\pi} u} 
\int\limits_{0}^{\infty} \exp(-v^2/u^2) \; dv \leqslant 1. 
\end{equation} 
If $t \in [0,T + \epsilon]$ or $t \in [2T - \epsilon,3T]$ then one has
\begin{align} 
&\Phi_{1}(t) \leqslant \frac{1}{\sqrt{\pi} u} 
\int\limits_{\sqrt{u}}^{\infty} 
\exp(-v^2/u^2) \; dv 
- \exp\Big((3T)^2 - \frac{1}{u}\Big) \exp(-(3T)^2) \nonumber \\
&= \frac{1}{2} \cdot \frac{2}{\sqrt{\pi}} 
\int\limits_{1/\sqrt{u}}^{\infty} 
\exp(-v^2) \; dv - \exp(-1/u) \nonumber \\
&\leqslant 1 \cdot \exp(-1/u) - \exp(-1/u) = 0.
\end{align}
If $t < 0$ then one has 
\begin{align} 
&\Phi_{1}(t) \leqslant \frac{1}{\sqrt{\pi} u} 
\int\limits_{\sqrt{u} - t}^{\infty} 
\exp(-v^2/u^2) \; dv 
- \exp\Big((3T)^2 - \frac{1}{u}\Big) \exp(-t^2) \nonumber \\
&= \frac{1}{2} \cdot \frac{2}{\sqrt{\pi}} 
\int\limits_{(\sqrt{u} - t)/u}^{\infty} 
\exp(-v^2) \; dv - \exp\Big((3T)^2 - \frac{1}{u} -t^2 \Big) \nonumber \\
&\leqslant 1 \cdot \exp\Big(-\{(\sqrt{u} - t)/u\}^2\Big) 
- \exp\Big(- \frac{1}{u} -t^2 \Big) \nonumber \\
&\leqslant \exp\Big( -1/u - (t/u)^2 \Big) 
- \exp\Big(- \frac{1}{u} -t^2 \Big) \leqslant 0.
\end{align}
Finally, if $t > 3T$ then one has 
\begin{align} 
&\Phi_{1}(t) \leqslant \frac{1}{\sqrt{\pi} u} 
\int\limits_{\sqrt{u} + t - 2T}^{\infty} 
\exp(-v^2/u^2) \; dv 
- \exp\Big((3T)^2 - \frac{1}{u}\Big) \exp(-t^2) \nonumber \\
&= \frac{1}{2} \cdot \frac{2}{\sqrt{\pi}} 
\int\limits_{(\sqrt{u} + t - 2T)/u}^{\infty} 
\exp(-v^2) \; dv - \exp\Big((3T)^2 - \frac{1}{u} -t^2 \Big) \nonumber \\
&\leqslant 1 \cdot \exp\Big(-\{(\sqrt{u} + t - 2T)/u\}^2\Big) 
- \exp\Big(- \frac{1}{u} -t^2 \Big) \nonumber \\
&\leqslant \exp\Big(-\{(\sqrt{u} + t/3)/u\}^2\Big) 
- \exp\Big(- \frac{1}{u} -t^2 \Big) \nonumber \\
&\leqslant \exp\Big( -1/u - (t/3u)^2 \Big) 
- \exp\Big(- \frac{1}{u} -t^2 \Big) \leqslant 0.
\end{align}

Now going back to (\ref{olikhets-lhs-delsteg-ett}), then clearly the LHS of Corollary \ref{wirtinger-ineq-scaled-sum-sum} is
\begin{equation} \label{olikhets-lhs-delsteg-tva}
\geqslant \displaystyle\sum_{q \geqslant 1} \; 
\Psi_{1} \Big(\frac{q}{Q}\Big) 
\int\limits_{-\infty}^{\infty} 
\Phi_{1}(t)
\sideset{}{^\flat}\sum_{\chi \textrm{(mod } q)} \; 
\Lambda_{A_{it},B_{-it}}(\chi) \; dt. 
\end{equation}
We apply Theorem \ref{sats1.3} to this expression and obtain 
\begin{equation} \label{olikhets-lhs-delsteg-tre}
a_3(\mathcal{L}) w(\Phi_{1}) 
\int\limits_{0}^{\infty} \frac{\Psi_{1}(x)x}{2} \; dx \; Q^{2} \mathcal{R}_{A,B}(Q) 
+ O\Big(Q^2 \log^8 Q (\log{\log Q})^M\Big).  
\end{equation}
At this point we simply claim that (\ref{olikhets-lhs-delsteg-tre}) equals
\begin{align} \label{olikhets-lhs-delsteg-fyra}
&C_{0}(\kappa) a_3(\mathcal{L}) Q^2 \log^9 Q 
\int\limits_{0}^{\infty} \frac{\Psi_{1}(x) x}{2} \; dx 
\int\limits_{-\infty}^{\infty} \Phi_{1}(t) 
|\Gamma((1/2+it)/2)|^6 \; dt \nonumber \\
&+ O\Big(Q^2 \log^8 Q (\log{\log Q})^M\Big), 
\end{align}
where $C_{0}(\kappa)$ is a coefficient\footnote{For the explicit expression of $C_{0}(\kappa),$ see (\ref{kappa-coef-1}).} which naturally depends on the shift $\kappa$. We will return to this in Section \ref{Diric-kapitel-calculation}. 

We now need to find lower bounds for $\Psi_{1}$ and $\Phi_{1}$. One easily shows that 
\begin{equation} 
\Psi_{1}(t) \geqslant \exp(-2u) \chi_{[5/4 + u,7/4 - u]}(t) 
\end{equation}
and
\begin{align} 
\Phi_{1}(t) 
&\geqslant \Big(1 - 2\exp\Big((3T)^2 - \frac{1}{u}\Big)\Big) 
\chi_{[T + \epsilon + 2\sqrt{u}, 2T - \epsilon - 2\sqrt{u}]} (t) \nonumber \\
&\; \; - \exp\Big((3T)^2 - \frac{1}{u}\Big)\exp(-t^2) 
\chi_{\mathbb{R} \backslash [T + \epsilon + 2\sqrt{u}, 2T - \epsilon - 2\sqrt{u}]}(t).  
\end{align}

Substituting these two into (\ref{olikhets-lhs-delsteg-fyra}), we conclude that (\ref{olikhets-lhs-delsteg-fyra}) is
\begin{align}
&\geqslant C_{1}(u,\epsilon) C_{0}(\kappa) a_3(\mathcal{L}) Q^2 \log^9 Q 
\int\limits_{5/4}^{7/4} \frac{x}{2} \; dx 
\int\limits_{T}^{2T} 
|\Gamma((1/2+it)/2)|^6 \; dt \nonumber \\
&\quad + O\Big(Q^2 \log^8 Q (\log{\log Q})^M\Big),
\end{align}
where $C_{1}(u,\epsilon)$ is some constant which tends to $1$ as $u$ and $\epsilon$ both tend\footnote{This is to be interpreted as that we can choose to use as small fixed $u$ and $\epsilon$ as we like.} to~$0$.

We can now conclude that the LHS of Corollary \ref{wirtinger-ineq-scaled-sum-sum} is 
\begin{align} \label{lhs-olikhet}
\geqslant C_{1}(u,\epsilon) C_{0}(\kappa) a_3(\mathcal{L}) Q^2 \log^9 Q 
&\int\limits_{5/4}^{7/4} \frac{x}{2} \; dx 
\int\limits_{T}^{2T} 
|\Gamma((1/2+it)/2)|^6 \; dt \nonumber \\
&+ O\Big(Q^2 \log^8 Q (\log{\log Q})^M\Big). 
\end{align}

\subsection{Upper bound for the RHS of Corollary \ref{wirtinger-ineq-scaled-sum-sum}} \label{Diric-kapitel-investigation-RHS}

Now let us turn our attention to the RHS of Corollary \ref{wirtinger-ineq-scaled-sum-sum}. We will deal with this in a similar way. The inequalities will go the other way around this time and we will make use of Theorem \ref{sats1.3-diff} rather than Theorem \ref{sats1.3}, but the basic strategy will be the same as in the treatment in Section \ref{Diric-kapitel-investigation-LHS} of the LHS. 

The function $f'(t,\chi,\kappa)^2$ features in the RHS of Corollary \ref{wirtinger-ineq-scaled-sum-sum}. Using the definition given in (\ref{f-def}) we find that 
\begin{align}
f'(t,\chi,\kappa) := &W'(t - \frac{2\pi\kappa}{\log Q}, \chi) W(t,\chi) 
W(t + \frac{2\pi\kappa}{\log Q}, \chi) \label{f'-expansion} \\ 
&+ W(t - \frac{2\pi\kappa}{\log Q}, \chi) W'(t,\chi) W(t + \frac{2\pi\kappa}{\log Q}, \chi) \nonumber \\
&+ W(t - \frac{2\pi\kappa}{\log Q}, \chi) W(t,\chi) 
W'(t + \frac{2\pi\kappa}{\log Q}, \chi). \nonumber
\end{align}
Hence we get nine different parts to consider. Now, forgetting for the moment that the integral in Corollary \ref{wirtinger-ineq-scaled-sum-sum} does not quite go from $T$ to $2T$ and that Theorem \ref{sats1.3-diff} requires $\Psi$ and $\Phi$ to be reasonable, let us see what the big picture is. Denoting the shifts in $A$ and $B$ by $\alpha_{1},$ $\alpha_{2},$ $\alpha_{3}$ and $\alpha_{4},$ $\alpha_{5},$ $\alpha_{6},$ we may apply partial differentiation in Theorem \ref{sats1.3-diff} to whichever pair of those six variables we like. Let us decide to after differentiation put 
$\alpha_{1} = \alpha_{4} = \frac{\kappa i}{\log Q},$ 
$\alpha_{2} = \alpha_{5} = 0$ and 
$\alpha_{3} = \alpha_{6} = -\frac{\kappa i}{\log Q}$.
Say for example that we wanted to do the case of 
\begin{equation}
W'(t - \frac{\kappa}{\log Q}, \chi)^2 
W(t,\chi)^2 W(t + \frac{\kappa}{\log Q}, \chi)^2. 
\end{equation} 
Then we differentiate with respect to $\alpha_{3}$ and $\alpha_{4}$. 
Naturally we use that differentiated case of Theorem \ref{sats1.3-diff}. Just as we in Corollary \ref{cor1} got the coefficient $42/9!,$ we will here get coefficients depending on $\kappa.$ The calculations of these $\kappa$-coefficients are elementary but take a lot of time. Indeed, they are similar to the calculation in Section \ref{Diric-kapitel-calculation}, where $C_{0}(\kappa)$ is found. However, before we go into this, we will now first explicitly show that effectively Theorem \ref{sats1.3-diff} may be applied.

We will now construct suitable functions $\Psi_{2}$ and $\Phi_{2}$. 

Construction of $\Psi_{2}:$
Recall the definition of the function $g(x)$ given in (\ref{g-def}).
Now let 
\begin{equation}
\Psi_{2}(t):= \exp(2u) g(t - (5/4 - u)) g((7/4 + u) - t). 
\end{equation} 
Then clearly $\Psi_{2}(t)$ is compactly supported in $[1,2]$ and 
\begin{equation} 
\Psi_{2}(t) \geqslant \chi_{(5/4,7/4]}(t). 
\end{equation}

Construction of $\Phi_{2}:$ 
\begin{equation}
\Phi_{2}(z) := \frac{u^{-1}}{\sqrt{\pi}(1 - \exp(-1/u))} 
\int\limits_{T - \sqrt{u}}^{2T + \sqrt{u}} 
\exp(-(v-z)^2/u^2) \; dv. 
\end{equation} 
For $t \in \mathbb{R}$ we have the following:
\begin{equation} 
\Phi_{2}(t) \geqslant \chi_{[T,2T]}(t). 
\end{equation}

Thus the RHS of Corollary \ref{wirtinger-ineq-scaled-sum-sum} is
\begin{equation} 
\leqslant \Big( \frac{\kappa}{\pi \log Q} \Big)^2 
\displaystyle\sum_{q \geqslant 1} \; \;
\Psi_{2} \Big(\frac{q}{Q}\Big)
\int\limits_{-\infty}^{\infty} 
\Phi_{2}(t) 
\sideset{}{^\flat}\sum_{\chi \textrm{(mod } q)} \; 
f'(t,\chi,\kappa)^2 \; dt.
\end{equation}
We write out the nine terms we get from expanding out $f'(t,\chi,\kappa)^2$ (cf.\ (\ref{f'-expansion})). In each case we may then apply Theorem \ref{sats1.3-diff}. Then it is crucial to find the $\kappa$-depending main-term-coefficients mentioned above. Before exploring what these coefficients are, we must conclude the present discussion. Thus after having done the above in any of the nine cases we are left with an expression of the type (ignoring for now the constant 
$( \frac{\kappa}{\pi \log Q} )^2$)
\begin{align} \label{olikhets-rhs-delsteg-tva}
&C_{i}(\kappa) a_3(\mathcal{L}) Q^2 \log^{11} Q 
\int\limits_{0}^{\infty} \frac{\Psi_{2}(x) x}{2} \; dx 
\int\limits_{-\infty}^{\infty} \Phi_{2}(t) 
|\Gamma((1/2+it)/2)|^6 \; dt \nonumber \\
&\; \; + O\Big(Q^2 \log^{10} Q (\log{\log Q})^M\Big), 
\end{align}
where $C_{i}(\kappa)$ is the relevant $\kappa$-constant in the $i$-th case out of the total nine cases.

We now need to find upper bounds for $\Psi_{2}$ and $\Phi_{2}$. 
For $t \in \mathbb{R}$ we have the following:
\begin{equation} 
\Psi_{2}(t) \leqslant \exp(2u) \chi_{[5/4 - u,7/4 + u]}(t) 
\end{equation}
and
\begin{align} 
\Phi_{2}(t) 
&\leqslant \frac{1}{(1 - \exp(-1/u))} 
\chi_{[T - 2\sqrt{u},2T + 2\sqrt{u}]}(t) \nonumber \\
&\quad + \exp\Big((3T)^2 - \frac{1}{u}\Big)\exp(-t^2) 
\chi_{\mathbb{R} \backslash [T - 2\sqrt{u},2T + 2\sqrt{u}]}(t).  
\end{align}

Substituting these two into (\ref{olikhets-rhs-delsteg-tva}), we conclude that (\ref{olikhets-rhs-delsteg-tva}) is
\begin{align} 
&\leqslant C_{i}(\kappa) a_3(\mathcal{L}) Q^2 \log^{11} Q 
\int\limits_{5/4 - u}^{7/4 + u} \frac{x \exp(2u)}{2} \; dx 
\; \nonumber \\
&\quad \times \int\limits_{-\infty}^{\infty} 
\bigg\{ \frac{1}{(1 - \exp(-1/u))} \chi_{[T - 2\sqrt{u},2T + 2\sqrt{u}]}(t) \nonumber \\
&\quad \quad + \exp\Big((3T)^2 - \frac{1}{u}\Big)\exp(-t^2) 
\chi_{\mathbb{R} \backslash [T - 2\sqrt{u},2T + 2\sqrt{u}]}(t) \bigg\}
|\Gamma((1/2+it)/2)|^6 \; dt \nonumber \\
&\quad + O\Big(Q^2 \log^{10} Q (\log{\log Q})^M\Big). \nonumber \\
&\leqslant C_{2}(u,\epsilon) C_{i}(\kappa) a_3(\mathcal{L}) Q^2 \log^{11} Q 
\int\limits_{5/4}^{7/4} \frac{x}{2} \; dx 
\int\limits_{T}^{2T} 
|\Gamma((1/2+it)/2)|^6 \; dt \nonumber \\
&\quad + O\Big(Q^2 \log^{10} Q (\log{\log Q})^M\Big), \label{olikhets-rhs-delsteg-tre}
\end{align}
where $C_{2}(u,\epsilon)$ is some constant which tends to $1$ as $u$ and $\epsilon$ both tend to $0$.

\subsection{$\kappa$-coefficients} \label{Diric-kapitel-investigation-kappa-coefficients}

We finish this section by stating what these $\kappa$-coefficients are. For simplicity let us use the notation Macl$(f(z)) = f(z) -$ Princ$(f(z)),$ where 
Princ$(f(z))$ is the principal part of $f(z)$. An example is 
Macl$(\frac{\cos z}{z^2}) = \frac{\cos z}{z^2} - \frac{1}{z^2}$. 
Let us introduce the following notation for the $\kappa$-coefficients:
\\
\\
$C_{1}(\kappa)$ in the case of \ldots 
\begin{equation}
W(t - \frac{2\pi\kappa}{\log Q}, \chi)^2 
W'(t,\chi)^2 W(t + \frac{2\pi\kappa}{\log Q}, \chi)^2, 
\end{equation} 
\\
\\
$C_{2}(\kappa) = C_{3}(\kappa)$ in the case of \ldots 
\begin{equation}
W'(t - \frac{2\pi\kappa}{\log Q}, \chi) 
W(t - \frac{2\pi\kappa}{\log Q}, \chi) W(t,\chi)^2 
W'(t + \frac{2\pi\kappa}{\log Q}, \chi) 
W(t + \frac{2\pi\kappa}{\log Q}, \chi), 
\end{equation} 
\\
\\
$C_{4}(\kappa) = C_{5}(\kappa)$ in the case of \ldots 
\begin{equation}
W'(t - \frac{2\pi\kappa}{\log Q}, \chi)^2 
W(t,\chi)^2 W(t + \frac{2\pi\kappa}{\log Q}, \chi)^2 
\end{equation}
and 
\begin{equation}
W(t - \frac{2\pi\kappa}{\log Q}, \chi)^2 
W(t,\chi)^2 W'(t + \frac{2\pi\kappa}{\log Q}, \chi)^2 
\end{equation} 
\\
\\
and finally $C_{6}(\kappa) = C_{7}(\kappa) = C_{8}(\kappa) = C_{9}(\kappa)$ in the case of \ldots 
\begin{equation}
W'(t - \frac{2\pi\kappa}{\log Q}, \chi) 
W(t - \frac{2\pi\kappa}{\log Q}, \chi) W(t,\chi) W'(t,\chi) 
W(t + \frac{2\pi\kappa}{\log Q}, \chi)^2 
\end{equation} 
and 
\begin{equation}
W(t - \frac{2\pi\kappa}{\log Q}, \chi)^2 W(t,\chi) W'(t,\chi) 
W'(t + \frac{2\pi\kappa}{\log Q}, \chi)
W(t + \frac{2\pi\kappa}{\log Q}, \chi). 
\end{equation} 

Thus the nine cases actually only give rise to four different $\kappa$-coefficients.

The calculations needed to find our $\kappa$-coefficients are long, but similar. We present the details of the simplest case, namely $C_{0}(\kappa),$ in Section \ref{Diric-kapitel-calculation}. We have 

\begin{equation} \label{kappa-coef-1}
C_{0}(\kappa) = \textrm{Macl}\Big\{ 
\frac{\cos\kappa}{\kappa^{8}} 
+ \frac{\sin\kappa}{\kappa^{9}} 
+ \frac{\cos(2\kappa)}{8\kappa^{8}} 
- \frac{\sin(2\kappa)}{2\kappa^{9}} 
\Big\},
\end{equation}

\begin{align} 
C_{1}(\kappa) 
=& \textrm{ Macl}\Big\{ 
- \frac{\cos\kappa}{4\kappa^{8}} 
+ \frac{3\sin\kappa}{4\kappa^{9}} 
+ \frac{\cos\kappa}{\kappa^{10}}
- \frac{\sin\kappa}{\kappa^{11}} \\
&+ \frac{\cos(2\kappa)}{96\kappa^{8}} 
- \frac{\sin(2\kappa)}{8\kappa^{9}} 
- \frac{\cos(2\kappa)}{2\kappa^{10}} 
+ \frac{\sin(2\kappa)}{2\kappa^{11}} 
\Big\}, \nonumber 
\end{align}

\begin{align} 
C_{2}(\kappa) 
=& \textrm{ Macl}\Big\{ 
- \frac{3\sin\kappa}{4\kappa^{9}} 
- \frac{11\cos\kappa}{4\kappa^{10}}
- \frac{21\sin\kappa}{4\kappa^{11}} \\
&+ \frac{\cos(2\kappa)}{32\kappa^{8}} 
- \frac{3\sin(2\kappa)}{8\kappa^{9}} 
- \frac{53\cos(2\kappa)}{32\kappa^{10}} 
+ \frac{21\sin(2\kappa)}{8\kappa^{11}} 
\Big\}, \nonumber 
\end{align}

\begin{align} 
C_{4}(\kappa) 
=& \textrm{ Macl}\Big\{ 
- \frac{\cos\kappa}{12\kappa^{8}} 
+ \frac{5\sin\kappa}{4\kappa^{9}} 
+ \frac{3\cos\kappa}{\kappa^{10}}
+ \frac{5\sin\kappa}{\kappa^{11}} \\
&- \frac{\cos(2\kappa)}{32\kappa^{8}} 
+ \frac{3\sin(2\kappa)}{8\kappa^{9}} 
+ \frac{13\cos(2\kappa)}{8\kappa^{10}} 
- \frac{5\sin(2\kappa)}{2\kappa^{11}} 
\Big\} \nonumber 
\end{align}
and
\begin{equation} 
C_{6}(\kappa) = \textrm{Macl}\Big\{ 
\frac{\cos\kappa}{8\kappa^{8}} 
- \frac{3\sin\kappa}{8\kappa^{9}} 
- \frac{5\cos\kappa}{4\kappa^{10}}
- \frac{3\sin\kappa}{4\kappa^{11}} 
- \frac{\cos(2\kappa)}{16\kappa^{10}} 
+ \frac{3\sin(2\kappa)}{8\kappa^{11}} 
\Big\}.
\end{equation}

\begin{remark} \label{42-magi}
Note that the $\kappa^0$-coefficient in the Taylor series of $C_{0}(\kappa)$ is $\frac{42}{9!}$ and that the $\kappa^0$-coefficient in the Taylor series of $C_{i}(\kappa)$ for $1 \leqslant i \leqslant 9$ is $\frac{3}{10!}$. These are namely the expected results if we let $\kappa \to 0$.
\end{remark}

\section{Conclusion of the proof of Theorem \ref{huvudsats}} \label{Diric-kapitel-conclusion}

Having done the hard work in Section \ref{Diric-kapitel-investigation}, we now just have to put things together. 
We have shown that the LHS of 
Corollary \ref{wirtinger-ineq-scaled-sum-sum} is (see (\ref{lhs-olikhet})) 
\begin{align} 
&\geqslant C_{1}(u,\epsilon) C_{0}(\kappa) a_3(\mathcal{L}) Q^2 \log^9 Q 
\int\limits_{5/4}^{7/4} \frac{x}{2} \; dx 
\int\limits_{T}^{2T} 
|\Gamma((1/2+it)/2)|^6 \; dt \nonumber \\
&+ O\Big(Q^2 \log^8 Q (\log{\log Q})^M\Big).  
\end{align}
We have shown that the RHS of 
Corollary \ref{wirtinger-ineq-scaled-sum-sum} is 
(see (\ref{olikhets-rhs-delsteg-tre}) and the comment just above (\ref{kappa-coef-1}) and also remember to once again include the constant 
$( \frac{\kappa}{\pi \log Q} )^2$ which we temporarily left out from our discussion before)
\begin{align} 
&\leqslant 
\Big( \frac{\kappa}{\pi \log Q} \Big)^2
C_{2}(u,\epsilon) 
\Big\{ C_{1}(\kappa) + 2C_{2}(\kappa) +2C_{4}(\kappa) +4C_{6}(\kappa) \Big\} a_3(\mathcal{L}) Q^2 \log^{11} Q \nonumber \\
&\quad \times \int\limits_{5/4}^{7/4} \frac{x}{2} \; dx 
\int\limits_{T}^{2T} 
|\Gamma((1/2+it)/2)|^6 \; dt 
+ O\Big(Q^2 \log^{8} Q (\log{\log Q})^M\Big). 
\end{align}

The above clearly implies that 
\begin{equation} \label{kappa-motsagelse-ekv}
C_{0}(\kappa) \leqslant \Big( \frac{\kappa}{\pi} \Big)^2 
\Big\{ C_{1}(\kappa) + 2C_{2}(\kappa) + 2C_{4}(\kappa) + 4C_{6}(\kappa) \Big\}.
\end{equation}  
Equality holds in (\ref{kappa-motsagelse-ekv}) when 
$\kappa \approx 7.42 \approx 1.18 \cdot 2\pi$.
If $\kappa$ is smaller than that value, then we do get a contradiction to (\ref{kappa-motsagelse-ekv}) and hence to the Main Assumption. 
In other words we must have that 
$\kappa > 7.42 \approx 1.18 \cdot 2\pi$. But this (a priori) immediately gives us our Main Theorem (Theorem \ref{huvudsats}).

\section{How to calculate the $\kappa$-coefficients} \label{Diric-kapitel-calculation}

Here we show (\ref{kappa-coef-1}), the calculation of the other 
$C_{i}(\kappa)$-coefficients being similar although slightly more difficult.

We begin by considering (\ref{olikhets-lhs-delsteg-tre}), keeping in mind that 
\begin{equation}
\mathcal{R}_{A,B}(Q)
= \displaystyle\sum_{\substack{S \subseteq A \\ T \subseteq B \\ |S| = |T|}} 
\mathcal{R}(\bar{S} \cup (-T),\bar{T} \cup (-S);Q), 
\end{equation} 
where 
\begin{equation}
\mathcal{R}(X,Y;Q) = Q^{\delta_{X,Y}} 
\prod_{\substack{x \in X \\ y \in Y}} 
\frac{1}{x + y} 
\end{equation}
and that we essentially want the shifts in the sets $A$ and $B$ to be 
$\frac{\kappa i}{\log Q},$ $0$ and $-\frac{\kappa i}{\log Q}$. 

One way to explicitly proceed is as follows. Let 
$A = \{ ai + \delta, bi + 2\delta, ci + 4\delta \}$ and 
$B = \{ di + \delta, ei + 2\delta, fi + 4\delta \},$
where $a = d = \frac{\kappa}{\log Q},$ $b = e = 0,$ 
$c = f = \frac{-\kappa}{\log Q}$ and $\delta$ is a small number which we will let tend to $0$. There are in total $20$ choices for the sets $S$ and $T$. It will turn out that we only have to look at $10$ of them in detail, by a symmetry-argument. To those $10$ pairs of $S$ and $T$ which we must consider the corresponding pairs of $\bar{S} \cup (-T)$ and $\bar{T} \cup (-S)$ are as follows:
\begin{equation} \label{appendix-kappa-1}
\{ ai + \delta, bi + 2\delta, ci + 4\delta \} \textrm{ \& } 
\{ di + \delta, ei + 2\delta, fi + 4\delta \}, 
\end{equation}
\begin{equation} \label{appendix-kappa-2}
\{ bi + 2\delta, ci + 4\delta, -di -\delta \} \textrm{ \& } 
\{ ei + 2\delta, fi + 4\delta, -ai - \delta \}, 
\end{equation}
\begin{equation} \label{appendix-kappa-3}
\{ bi + 2\delta, ci + 4\delta, -ei - 2\delta \} \textrm{ \& } 
\{ di + \delta, fi + 4\delta, -ai - \delta \}, 
\end{equation}
\begin{equation} \label{appendix-kappa-4}
\{ bi + 2\delta, ci + 4\delta, -fi - 4\delta \} \textrm{ \& } 
\{ di + \delta, ei + 2\delta, -ai - \delta \}, 
\end{equation}
\begin{equation} \label{appendix-kappa-5}
\{ ai + \delta, ci + 4\delta, -di - \delta \} \textrm{ \& } 
\{ ei + 2\delta, fi + 4\delta, -bi - 2\delta \}, 
\end{equation}
\begin{equation} \label{appendix-kappa-6}
\{ ai + \delta, ci + 4\delta, -ei - 2\delta \} \textrm{ \& } 
\{ di + \delta, fi + 4\delta, -bi - 2\delta \}, 
\end{equation}
\begin{equation} \label{appendix-kappa-7}
\{ ai + \delta, ci + 4\delta, -fi - 4\delta \} \textrm{ \& } 
\{ di + \delta, ei + 2\delta, -bi - 2\delta \}, 
\end{equation}
\begin{equation} \label{appendix-kappa-8}
\{ ai + \delta, bi + 2\delta, -di - \delta \} \textrm{ \& } 
\{ ei + 2\delta, fi + 4\delta, -ci - 4\delta \}, 
\end{equation}
\begin{equation} \label{appendix-kappa-9}
\quad \quad \quad \{ ai + \delta, bi + 2\delta, -ei - 2\delta \} \textrm{ \& } 
\{ di + \delta, fi + 4\delta, -ci - 4\delta \} \textrm{ and}
\end{equation}
\begin{equation} \label{appendix-kappa-10}
\{ ai + \delta, bi + 2\delta, -fi - 4\delta \} \textrm{ \& } 
\{ di + \delta, ei + 2\delta, -ci - 4\delta \}. 
\end{equation}

In each of the above listed cases (and in the other $10$ cases although we never have to do the latter ones in ``practice") we will treat the relevant $\mathcal{R}(X,Y;Q)$ by representing both $Q^{\delta_{X,Y}}$ and
$\displaystyle\prod_{\substack{x \in X \\ y \in Y}} \frac{1}{x + y}$ by their Laurent series and multiply them together, seeing $\delta$ as our variable. We then proceed by summing up all the $20$ Laurent series to obtain one final Laurent series. Since (\ref{olikhets-lhs-delsteg-tva}) is a continuous function of $\delta$ as $\delta \to 0,$ so must (\ref{olikhets-lhs-delsteg-tre}) be. This means that the Laurent series of $\mathcal{R}_{A,B}(Q)$ which we had obtained can not have any negative $\delta$-powers. Furthermore, we are uninterested in the positive $\delta$-powers as we will let $\delta \to 0$. We therefore focus on identifying the $\delta^0$-coefficient and therefore in turn on finding the $\delta^0$-coefficient in each of the $20$ cases. Below we will discuss how to do the $10$ cases listed above, before explaining why we then get the other $10$ cases for ``free".

These $\delta^0$-coefficients will be expressions in terms of $\kappa$. We will treat them as Laurent series. Each individual such Laurent series may have negative $\kappa$-powers but the sum of all $20$ Laurent series is an analytic expression in $\kappa,$ since (\ref{olikhets-lhs-delsteg-tva}) and therefore (\ref{olikhets-lhs-delsteg-tre}) is a continuous function of $\kappa$ as 
$\kappa \to 0$. We here remark that this sum will be some constant in terms of $\kappa,$ namely $C_{0}(\kappa)$ by definition, multiplied by $\log^9 Q$. We will then substitute this expression for $\mathcal{R}_{A,B}(Q)$ into (\ref{olikhets-lhs-delsteg-tre}). Because of the cancellation of possible negative $\kappa$-powers we will below in each of the $20$ cases focus on finding only the analytic part of the $\kappa$-expressions. The sum of these is thus guaranteed to equal $C_{0}(\kappa)$.

Upon inspection of (\ref{appendix-kappa-1})-(\ref{appendix-kappa-10}) we notice that (\ref{appendix-kappa-1}), (\ref{appendix-kappa-4}), (\ref{appendix-kappa-6}) and (\ref{appendix-kappa-8}) will then not give any contribution to $C_{0}(\kappa),$ since in those cases $\kappa$ will not feature in $\delta_{X,Y}$ (the latter means that we will only get negative $\kappa$-powers). Moreover, clearly the contribution from the terms (\ref{appendix-kappa-3}) and (\ref{appendix-kappa-5}) will be equal and likewise the contribution from (\ref{appendix-kappa-7}) will be the same as from (\ref{appendix-kappa-9}).

Let us now in some detail find the contribution to $C_{0}(\kappa)$ from the term (\ref{appendix-kappa-2}) above. The basic plan is to write out what we get from multiplying the Laurent series of 
$Q^{\delta_{X,Y}}$ and 
$\displaystyle\prod_{\substack{x \in X \\ y \in Y}} \frac{1}{x + y}$ and then let $\delta \to 0$. This will give us an expression in terms of $\kappa$ and by eliminating any negative $\kappa$-powers, we obtain the contribution in this case to $C_{0}(\kappa)$. Here are the details:
\begin{align}
&\exp\Big(\Big(\frac{2\kappa i}{\log Q} - 5\delta\Big)\log Q\Big) \cdot 
\frac{1}{\big(\frac{2\kappa i}{\log Q} + 2\delta\big)} \cdots 
\frac{1}{\big(\frac{2\kappa i}{\log Q} - 8\delta\big)} \nonumber \\
&= \exp(2\kappa i) \cdot \exp(-5\delta \log Q) \cdot 
\frac{\log^{8}Q}{-64\delta \kappa^8} \cdot 
\frac{1}{\big(1 - \frac{20\delta \log Q}{i\kappa} + \ldots \big)} \nonumber \\
&= \exp(2\kappa i) \cdot \exp(-5\delta \log Q) \cdot 
\frac{\log^{8}Q}{-64\delta \kappa^8} \cdot 
\Big(1 + \frac{20\delta \log Q}{i\kappa} + \ldots \Big).
\end{align}
Recall that the pair of sets corresponding to (\ref{appendix-kappa-2}) comes from a pair of $S$ and $T$. Had we instead done the above calculation in the case corresponding to the pair of sets coming from the pair $\bar{S}$ and $\bar{T}$ we would have obtained the same expression, but with a minus-sign in front and with negative exponentials. Putting these two 
together\footnote{One can for each pair of sets in the list (\ref{appendix-kappa-1})-(\ref{appendix-kappa-10}) make the analogous choice of a pair of sets among the ``other" $10$ pairs of sets. Then they will together make a contribution that can be expressed in trigonometric functions.}
and focusing on the $\delta^0$-coefficient we obtain
\begin{align} 
&\frac{2 \cdot \cos(2\kappa) \cdot (-5) \cdot \log^{9}Q}{(-64) \cdot \kappa^{8}} 
+ \frac{2 \cdot \sin(2\kappa) \cdot 20 \cdot \log^{9}Q}{(-64) \cdot \kappa^{9}} \nonumber \\
&= \log^{9}Q \cdot \Big\{ \frac{5\cos(2\kappa)}{32\kappa^8} 
- \frac{5\sin(2\kappa)}{8\kappa^9} \Big\}.
\end{align}
Again, recall that our strategy is to ignore any negative $\kappa$-powers, hence the introduction of the ``Macl"-notation used in this text in Section \ref{Diric-kapitel-investigation-kappa-coefficients}. Thus our final answer (and to make notation here coherent with the one in Section \ref{Diric-kapitel-investigation-kappa-coefficients} we exclude the 
``$\log^{9} Q$"-term) for the contribution from these two (out of the total $20$) pairs equals 
\begin{equation} 
\textrm{Macl}\Big\{ \frac{5\cos(2\kappa)}{32\kappa^8} 
- \frac{5\sin(2\kappa)}{8\kappa^9} \Big\}.
\end{equation}
It is possible to handle the contribution coming from (\ref{appendix-kappa-3}), (\ref{appendix-kappa-7}) and (\ref{appendix-kappa-10}) in an analogous way and one finds that the contribution in each of these three cases is respectively
\begin{equation} 
\textrm{Macl}\Big\{ \frac{2\cos\kappa}{5\kappa^8} 
- \frac{19\sin\kappa}{20\kappa^9} \Big\},
\end{equation}
\begin{equation} 
\textrm{Macl}\Big\{ \frac{\cos\kappa}{10\kappa^8} 
+ \frac{29\sin\kappa}{20\kappa^9} \Big\}
\end{equation}
and
\begin{equation} 
\textrm{Macl}\Big\{ - \frac{\cos(2\kappa)}{32\kappa^8} 
+ \frac{\sin(2\kappa)}{8\kappa^9} \Big\}.
\end{equation}

Hence, putting things together we have that
\begin{align}
C_{0}(\kappa) &= \textrm{Macl}\Big\{ 
\Big( \frac{5\cos(2\kappa)}{32\kappa^8} 
- \frac{5\sin(2\kappa)}{8\kappa^9} \Big) 
+ 2 \cdot \Big( \frac{2\cos\kappa}{5\kappa^8} 
- \frac{19\sin\kappa}{20\kappa^9} \Big) \nonumber \\
&\quad + 2 \cdot \Big( \frac{\cos\kappa}{10\kappa^8} 
+ \frac{29\sin\kappa}{20\kappa^9} \Big) 
+ \Big( - \frac{\cos(2\kappa)}{32\kappa^8} 
+ \frac{\sin(2\kappa)}{8\kappa^9} \Big) \Big\} \nonumber \\
&= \textrm{Macl}\Big\{ 
\frac{\cos\kappa}{\kappa^{8}} 
+ \frac{\sin\kappa}{\kappa^{9}} 
+ \frac{\cos(2\kappa)}{8\kappa^{8}} 
- \frac{\sin(2\kappa)}{2\kappa^{9}} 
\Big\}.
\end{align}

\end{document}